\documentclass[a4paper, 10pt, twoside, english]{article}

\usepackage[utf8]{inputenc}
\usepackage[T1]{fontenc}
\usepackage[sc]{mathpazo}
\usepackage[english]{babel}
\usepackage{geometry}
\usepackage{amsmath,amsfonts,amssymb,amsthm}
\usepackage[final,colorlinks]{hyperref}
\usepackage{graphicx}
\usepackage{cleveref}
\hypersetup{
     colorlinks=true,
     linkcolor=black,
     filecolor=blue,
     urlcolor=Mahogany,
     }
\usepackage{mathtools}
\usepackage[dvipsnames]{xcolor}
\usepackage{pdflscape}
\usepackage{etoolbox}
\usepackage{fancyhdr}
\usepackage{lastpage}
\usepackage{ifdraft}
\usepackage[autostyle,italian=guillemets]{csquotes}
\usepackage[backend=bibtex, maxnames=50, firstinits=false, sorting=nyt]{biblatex}
\usepackage{colortbl}
\usepackage{booktabs}
\usepackage{hyphenat}
\usepackage{siunitx}
\usepackage[mathlines,pagewise]{lineno}

\graphicspath{{./images/}}

\newcommand*\patchAmsMathEnvironmentForLineno[1]{%
  \expandafter\let\csname old#1\expandafter\endcsname\csname #1\endcsname
  \expandafter\let\csname oldend#1\expandafter\endcsname\csname end#1\endcsname
  \renewenvironment{#1}%
  {\linenomath\csname old#1\endcsname}%
  {\csname oldend#1\endcsname\endlinenomath}}%
\newcommand*\patchBothAmsMathEnvironmentsForLineno[1]{%
  \patchAmsMathEnvironmentForLineno{#1}%
  \patchAmsMathEnvironmentForLineno{#1*}}%
\patchBothAmsMathEnvironmentsForLineno{equation}%
\patchBothAmsMathEnvironmentsForLineno{align}%
\patchBothAmsMathEnvironmentsForLineno{flalign}%
\patchBothAmsMathEnvironmentsForLineno{alignat}%
\patchBothAmsMathEnvironmentsForLineno{gather}%
\patchBothAmsMathEnvironmentsForLineno{multline}%

\newcommand{\dd}{\mathop{}\!\mathrm{d}}

\DeclareMathOperator{\Span}{Span}
\DeclareMathOperator{\dist}{dist}

\DeclareMathOperator{\Range}{Range}

\newcommand{\R}{\mathbb{R}}
\newcommand{\N}{\mathbb{N}}

\newcommand{\C}{\mathbb{C}}

\theoremstyle{plain}
\newtheorem{theorem}{Theorem}[section]

\newtheorem{proposition}[theorem]{Proposition}
\newtheorem{corollary}[theorem]{Corollary}

\theoremstyle{definition}
\newtheorem{remark}[theorem]{Remark}
\newtheorem{assumption}{Assumption}

\numberwithin{equation}{section}
\numberwithin{table}{section}
\numberwithin{figure}{section}

\usepackage[shortlabels]{enumitem}
\setlist[enumerate,1]{label=\textnormal{(\emph{\roman*})}}

\fancypagestyle{mypagestyle}{%
  \fancyhf{}%
  \fancyhead[LO,RE]{\scriptsize Stability of age-structured models with nonlocal diffusion}%
  \fancyhead[RO,LE]{\scriptsize D. Breda, S. De Reggi, R. Vermiglio}%
  \fancyfoot[LE,RO]{\scriptsize \thepage\,/\,\pageref*{LastPage}}%
}
\usepackage{authblk}
\usepackage{titling}

\usepackage{footnote}
\thanksmarkseries{arabic}

\title{A numerical method for  the stability analysis of linear age-structured models with nonlocal diffusion\thanks{\textbf{PRIN}: this work was partially supported by the Italian Ministry of University and Research (MUR) through the PRIN 2020 project (No.\ 2020JLWP23) ``Integrated Mathematical Approaches to Socio-Epidemiological Dynamics'' (CUP: E15F21005420006).}}

\author{Dimitri Breda\thanks{CDLAb - Computational Dynamics Laboratory, Department of Mathematics, Computer Science and Physics, University of Udine}\thanksgap{2mm}$^,$\thanks{\url{dimitri.breda@uniud.it}}
,\ Simone De Reggi\thanksmark{2}\thanksgap{2mm}$^,$\thanks{\url{simone.dereggi@uniud.it}},\ Rossana Vermiglio\thanksmark{2}\thanksgap{2mm}$^,$\thanks{\url{rossana.vermiglio@uniud.it}}}
\date{September 04, 2023}

\begin{document}
\maketitle
\begin{abstract}
We numerically address the stability analysis of linear age-structured population models with nonlocal diffusion, which arise naturally in describing dynamics of infectious diseases.
Compared to Laplace diffusion, models with nonlocal diffusion are more challenging since the associated semigroups have no regularizing properties in the spatial variable.  Nevertheless,  the asymptotic stability of the null equilibrium is determined by the spectrum of the infinitesimal generator associated to the semigroup.
We propose a numerical method to approximate the leading part of this spectrum by first reformulating the problem via integration of the age-state and then by discretizing the generator combining a spectral projection in space with a pseudospectral collocation in age. A rigorous convergence analysis proving spectral accuracy is provided in the case of separable model coefficients. Results are confirmed experimentally and numerical tests are presented also for the more general instance.  


\end{abstract}
\section{Introduction}\label{Introduction}
Population dynamics are often described by taking into consideration physiological characteristics of individuals \cite{auger2008structured}. Among these, age structure and  spatial movement represent some of the most important features to consider in the spread of an infectious disease \cite{iannelli1995mathematical, inaba2017age, kang2021mathematical}.

Age-structured epidemic models with spatial diffusion are often formulated by means of partial differential equations in which the diffusive term is represented by the Laplace operator (see \cite{cusulin2006diffusion} and the references in \cite{kang2020age}).
However, the latter describes the random spread of individuals in adjacent spatial positions, thus it is not suitable for describing, e.g., the spread of a disease via long-distance traveling \cite{kuniya2018global}.  
In recent years, many authors \cite{bates2007existence, kao2010random, kuniya2018global, liu2015nonlocal, xu2021spatial, yang2019dynamics, zhao2018spatial} have considered nonlocal (convolution) diffusion operators of the form 
\begin{equation}\label{conv3}
\int_{\Omega}J(x-y)\left[u(y)-u(x)\right]\dd y
\end{equation}
in order to describe the long-distance dispersal of a population $u$ in a region $\Omega\subseteq \R^n$, where $J(x-y)$ is interpreted as the probability of jumping from position $y\in\Omega$ to $x\in\Omega$ \cite{garcia2009principal}. Thus, in \eqref{conv3}, the convolution $\int_{\Omega}J(x-y)u(y)\dd y$ represents the rate at which individuals are arriving at position $x\in\Omega$ from other places while $\int_{\Omega}J(y-x)u(x)\dd y$ represents the rate at which individuals are leaving location $x$ to travel to other sites.

These and similar models have been intensively studied in a series of recent papers \cite{ducrot2022age, kang2021mathematical, kang2021nonlinear, kang2022principal, kang2020age, webb1987operator}, where the authors developed the relevant theory of the associated semigroups of solution operators, their asymptotic behavior, their spectral properties, as well as asynchronous exponential growth and nonlinear dynamics.  Moreover, in \cite{kang2021approximation} the authors investigated how age-structured models with Laplace diffusion can be approximated by models involving nonlocal diffusion.
All these aspects are well understood when the evolution is considered on abstract spaces, leading to deal with infinite-dimensional dynamical systems. Consequently, when practical computations are required, numerical approximations enter the scene. 
 
\bigskip
In this work, in view of investigating the stability of linear age-structured population models with nonlocal diffusion of Dirichlet type, we propose to approximate the (leading part of the) spectrum of the infinitesimal generator associated to the relevant semigroup by reduction to finite dimension. This is achieved by combining a spectral projection in space with a pseudospectral collocation in age after integrating the state on the age interval in order to gain on regularity. Then we also extend to the Neumann case. For previous works on age-structured models with no or Laplace diffusion see \cite{ando2022pseudospectral, breda2007stability, breda2013numerical, breda2008stability, breda2006pseudospectral, breda2012computing}.

The main contributions consist in rigorously proving the convergence of the approximating eigenvalues to the exact ones in the case of separable model coefficients and in experimentally validating the obtained results on some test cases. Further numerical tests are presented also for the general instance. 

\bigskip
The work is organized as follows: in \cref{section2} we present the prototype age-structured model with Dirichlet nonlocal diffusion of interest and the relevant assumptions.
In \cref{Abstract setting} we illustrate the abstract setting and the reformulation via integration of the age-state.
The numerical approach is proposed in \cref{sezioneapproccio}. In \cref{sezdim} we develop the convergence analysis. In \cref{secimplementation} we provide some details about the implementation of the method, while \cref{numresults} contains numerical experiments. In \cref{neumann} we discuss some results for models with Neumann nonlocal diffusion. Finally, in \cref{conclusions} we provide some concluding remarks.

MATLAB demos are available on GitHub via \url{http://cdlab.uniud.it/software}.

\section{Age-structured models with Dirichlet nonlocal diffusion}\label{section2}
Let $u(t,a , x)$  denote the density of a population at time $t$, age $a\in\left[0, a^\dagger\right]$ and position $x\in\Omega$, where $t\ge 0$, $a^\dagger\in(0, \infty)$ and $\Omega\subseteq\R^n$ is open, bounded and connected with $C^1$ boundary.

Age-structured models with Dirichlet nonlocal diffusion take the form \cite{kang2020age}
\begin{equation}\label{modellononloc} 
\left\{\setlength\arraycolsep{0.1em}\begin{array}{rlll} 
\mathcal Du(t, a, x)&=&\displaystyle d\int_{\R^n}J(x-y)[u(t, a, y)\\[3mm]&&-u(t, a, x)]\dd y \\[2mm]&&-\mu(a, x) u(t, a, x),& t>0,\  a\in[0, a^\dagger],\ x\in\Omega,\\[2mm]
u(t, 0, x)&=&\displaystyle \int_0^{a^\dagger} \beta(\hat  a, x)u(t,\hat  a, x)\dd\hat  a,\quad &t>0,\ 
x\in\Omega,\\[3mm]
u(t, a, x)&=&0, & t>0,\ a\in[0, a^\dagger], \  x\in\R^n\setminus	\Omega,\\[2mm]
u(0, a, x)&=&u_0(a, x),& a\in[0, a^\dagger],\ x\in\Omega,
\end{array} 
\right. 
\end{equation}
for
\begin{equation}\label{grad}
\mathcal Du(t, a, x):=\partial_t u(t, a, x)+\partial_a u(t, a, x),
\end{equation}
and $\beta, \mu\colon [0, a^\dagger]\times \Omega \to \R$ and $d>0$ describing birth, mortality and diffusion, respectively. The kernel $J\colon \R^n\to \R$ is a $C^0$, compactly supported\footnote{The hypothesis that $J$ has compact support can be actually relaxed if one supposes $\Omega$ to be bounded \cite{kang2022principal}.} nonnegative radial function satisfying 
\begin{equation*}
\int_{\R^n}J(x)\dd x=1, \quad J(x)\ge 0,\quad J(0)>0.
\end{equation*}
Note that $J(x-y)=J(y-x)$ and  $\int_{\R^n}J(x-y)\dd y=1$ for \emph{a.a.} $x\in\R^n$.
Moreover, the diffusion takes place in the whole $\R^n$, but $u$ vanishes outside $\Omega$,
thus the integrals in \eqref{modellononloc} can be considered in $\Omega$ instead of $\R^n$.
As common in the literature \cite{garcia2009principal, inaba2017age}, $u$ is an $L^1$-function in age and  an $L^2$-function in space, thus we introduce the space $U:=L^1([0, a^\dagger], L^2(\Omega,\R))$\footnote{For simplicity we write $u(t, a, x)$ for $u(t)(a)(x)$ and we adopt a similar notation hereafter. Moreover, in general, when we write $(a, x)\in [0, a^\dagger]\times \Omega$, we implicitly mean a.e..} equipped with the norm 
\begin{equation*}
\|\phi \|_U:=\int_0^{a^\dagger} \|\phi(a, \cdot)\|_{L^2(\Omega,\R)}\dd a.
\end{equation*} 
Following these observations, given $u_0\in U$, we say that $u$ is a mild solution of \eqref{modellononloc} if $u\in C^0([0, +\infty), U)$,  $u$ is extended by zero in $\R^n\setminus\Omega$ and satisfies
\begin{equation}\label{model} 
\left\{\setlength\arraycolsep{0.1em}\begin{array}{rcll} 
\mathcal Du(t, a, x)&=&d\left(\mathcal Ju-u\right)(t, a, x)\\[2mm]&&-\mu(a, x) u(t, a, x),& t>0,\ a\in[0, a^\dagger],\ x\in\Omega,\\[2mm]
u(t, 0, x)&=&\displaystyle \int_0^{a^\dagger} \beta(\hat  a, x)u(t,\hat  a, x)\dd\hat  a,\quad& t>0,\ x\in\Omega,\\[3mm]
u(0, a, x)&=&u_0(a, x),& a\in [0, a^\dagger],\ x\in\Omega,
\end{array} 
\right. 
\end{equation}
in the sense of \cite[Theorem 2.3]{kang2020age} where, for brevity, we introduce $\mathcal J\colon U\to U$ defined as
\begin{equation*}
\mathcal J\phi(a, x):=\int_{\Omega}J(x-y)\phi(a, y)\dd y.
\end{equation*}
Also, from the biological point of view, we write 
\begin{equation*}
\mu(a, x)=\mu_0(a)+\mu_{1}(a, x),
\end{equation*}
where $\mu_0$ represents the natural mortality and $\mu_{1}$ describes how the spatial position of individuals affects the mortality depending on the age.

The following requirements are necessary to proceed with the analysis \cite{ducrot2022age}.
\begin{assumption}\label{Assumption 1}
\hspace{0mm}
\begin{enumerate}
\item \label{ass1i}  $\mu_0\in L^1_{\text{loc}}([0, a^\dagger), \R)$ is positive and
$\int_0^{a^\dagger}\mu_0(a) \dd a=+\infty$;
\item \label{ass1ii} $\mu_{1}\in L^{\infty}\left([0, a^\dagger],C\left(\overline\Omega,\R\right)\right)$ is nonnegative;
\item \label{ass1iii} $\beta\in L^{\infty}\left([0, a^\dagger],C\left(\overline\Omega,\R\right)\right)$ is nonnegative.
\end{enumerate} 
\end{assumption}
By defining $\tilde u$ through $u(t, a, x)=\Pi_0(a) \tilde u(t, a, x)$ \cite{iannelli1995mathematical, inaba2017age} for the natural survival probability
\begin{equation*}
\Pi_0(a):=\exp\left(-\int_0^{a}\mu_0(\hat a)\dd \hat a\right),
\end{equation*}
\eqref{model} becomes

\begin{equation}\label{modelb} 
\left\{\setlength\arraycolsep{0.1em}\begin{array}{rcll} 
\mathcal D\tilde u(t, a, x)&=&d\left(\mathcal J\tilde u-\tilde u\right)(t, a, x)\\[2mm]&&-\mu_{1}(a, x) \tilde u(t, a, x),& t>0,\ a\in[0, a^\dagger],\ x\in\Omega,\\[2mm]
\tilde u(t, 0, x)&=&\displaystyle\int_0^{a^\dagger} \beta(\hat  a, x)\Pi_0(\hat  a)\tilde u(t,\hat  a, x)\dd\hat  a,\quad& t>0,\ x\in\Omega,\\[3mm]
\tilde u(0, a, x)&=&\tilde u_0(a, x),& a\in [0, a^\dagger],\ x\in\Omega,
\end{array} 
\right. 
\end{equation}
where $\tilde u_0:=\Pi_0^{-1} u_0$.
In view of asymptotic stability, \eqref{model} and \eqref{modelb} are equivalent since $\Pi_0$ is bounded thanks to \cref{Assumption 1} \ref{ass1i}.

\section{Abstract setting}\label{Abstract setting}
According to \cite{kang2021nonlinear, kang2020age}, 
\eqref{modelb} can be seen as the abstract Cauchy problem 
\begin{equation}\label{model2} 
\left\{\setlength\arraycolsep{0.1em}\begin{array}{rcll} 
\tilde u'(t)&=&\mathcal A\tilde u(t),& t>0,\\[2mm]
\tilde u(0)&=&\tilde u_0 \in U,&
\end{array} 
\right. 
\end{equation}
where $\mathcal A\colon D(\mathcal A)\subseteq U\to U$ is defined by
\begin{equation*}
\mathcal A\phi(a, x):=d\left(\mathcal J\phi-\phi\right)(a, x)-\partial_a \phi(a, x)-\mu_{1}(a, x) \phi(a, x),
\end{equation*}
\begin{equation}\label{condomain}
D(\mathcal A):=\left\{\phi\in U\ |\ \partial_a\phi\in U\ \text{and}\  \phi(0,x)=\int_0^{a^\dagger} \beta(\hat a, x)\Pi_0(\hat a)\phi(\hat  a, x)\dd\hat  a\right\}.
\end{equation}
Note that if $\tilde u_0\in D(\mathcal A)$ then $\tilde u$ is a classical solution in the sense of \cite[Theorem 2.3]{kang2020age}.
Under \cref{Assumption 1}, $\mathcal A$ generates the strongly continuous semigroup of bounded positive linear solution operators $\{T(t)\}_{t\ge 0}$ in $U$ given by $T(t)\tilde u_0 :=\tilde u(t)$. Moreover, its spectrum determines the stability of the null solution since the spectral abscissa $s(\mathcal A)$ of $\mathcal A$ coincides with the growth bound of the generated semigroup.
However, unlike the case of Laplace diffusion, $\{T(t)\}_{t\ge 0}$ is not eventually compact. This makes the stability analysis more involved since the spectrum of $\mathcal A$ is not necessarily just point spectrum. 
Nevertheless, sufficient conditions on the model coefficients are given in \cite{ducrot2022age}  (see definition (3.2) and Theorem 4.8 therein, to which we refer the interested reader for a complete illustration) for the existence of a \emph{principal} eigenvalue, i.e., a simple isolated real eigenvalue $\lambda_0$ strictly to the right of all the other spectral values of $\mathcal A$ (and with positive eigenfunction). Hence, under such conditions, $s(\mathcal A)=\lambda_0$ and thus everything reduces to computing the latter.  

\bigskip
In view of analyzing the stability of \eqref{modelb}, based on the above considerations, we propose to numerically approximate the principal eigenvalue $\lambda_0$ of $\mathcal A$. To this aim, inspired by \cite{ando2022pseudospectral, scarabel2021numerical}, we introduce the following equivalent reformulation which allows one to gain in regularity and to better deal with the condition characterizing newborns in $D(\mathcal A)$. 

Let $v(t):=\mathcal S\tilde u(t)$ for $\mathcal S\colon  U\to V$ defined as
\begin{equation*}
\mathcal S\phi(a, x):=\int_0^a\phi(\hat a, x)\dd\hat a
\end{equation*}
and
\begin{gather*}
V:=\left\{\psi\in U\ \bigg|\ \psi(a, x)=\int_0^{a}\phi(\hat a, x)\dd\hat a\text{ for some }\phi\in U\right\}
\end{gather*}
equipped with  $\|\psi\|_V:=\|\partial_a\psi\|_U$ (note that $\psi(0, x)=0$ a.e.  in $\Omega$).
$\mathcal S$ determines an isomorphism between $U$ and $V$ with inverse $\mathcal S^{-1}\psi=\partial_a \psi$.
Then $\tilde u$ is a solution of \eqref{model2} in $U$ if and only if  $v$ is a solution of 
\begin{equation*}
\left\{\setlength\arraycolsep{0.1em}\begin{array}{rcll} 
v'(t)&=&\mathcal Bv(t),&  t>0,\\[2mm]
v(0)&=&v_0 \in V,&
\end{array} 
\right. 
\end{equation*}
in $V$, where $v_0:=\mathcal S\tilde u_0$ and $\mathcal B:=\mathcal S\mathcal A\mathcal S^{-1}$ with domain $D(\mathcal B)=\mathcal SD(\mathcal A)\subset V$. Explicitly,
\begin{equation}\label{split}
\mathcal B\psi(a, x)=d(\mathcal J\psi-\psi)(a, x)+\mathcal K\psi(a, x)
\end{equation}
for $\mathcal K\colon D(\mathcal B)\to V$ given by
\begin{equation*}
\mathcal K\psi(a, x):=-\partial_a\psi(a, x)+\int_0^{a^\dagger}\beta(\hat a, x)\Pi_0(\hat a)\partial_a\psi(\hat a, x)\dd\hat a
-\int_0^{a}\mu_{1}(\hat a, x)\partial_a\psi(\hat a, x)\dd\hat a
\end{equation*}
and 
\begin{equation*}
D(\mathcal B)=\left\{\psi\in V\ |\ \partial_a\psi\text{ is absolutely continuous in $a$}\right\}.
\end{equation*}
Since $\sigma(\mathcal B)=\sigma(\mathcal A)$ \cite[Section 2.1]{engel2000one} and $\sigma_p(\mathcal B)=\sigma_p(\mathcal A)$ \cite[Proposition 4.1]{breda2012approximation}, in order to approximate $\lambda_0$  we propose to discretize $\mathcal B$ instead of $\mathcal A$, given that the domain condition for $\mathcal B$ is now trivial
and the original boundary condition on the newborns has moved to the action of $\mathcal K$.\footnote{Alternatively, the approach of \cite{breda2007stability, breda2008stability, breda2012computing} relies on a direct numerical treatment of the condition in $D(\mathcal A)$, which requires additional assumptions on $\beta$ and $\mu$.}

\subsection{Rates independent of the space variable}\label{specialcase}
In this section we work under the hypothesis that $\beta$ and $\mu_1$ depend only on age. 
\begin{assumption}\label{assumptionrates}
$\beta(a, x)=\beta(a)$ and $\mu_{1}(a, x)=\mu_{1}(a)$ for a.a. $x\in\Omega$.
\end{assumption}
\noindent  It follows from \cite{kang2020age} that $\sigma(\mathcal A)= \sigma_p(\mathcal A)$. Consequently, $\sigma(\mathcal B)= \sigma_p(\mathcal B)$. Moreover, as we illustrate next, $\sigma(\mathcal B)$ can be characterized through separate eigenvalue problems in age and space, respectively. 

\bigskip
To this aim, as far as age-dependency is considered, let $\mathcal K^a\colon D(\mathcal K^a)\subseteq V^a\to V^a$ for
\begin{equation*}
\mathcal K^af(a):=-f'(a)+\int_0^{a^\dagger}\beta(\hat a)\Pi_0(\hat a)f'(\hat a)\dd\hat a
-\int_0^{a}\mu_{1}(\hat a)f'(\hat a)\dd\hat a,
\end{equation*}
\begin{equation*}
D(\mathcal K^a):=\left\{f\in V^a\ |\ f'\text{ is absolutely continuous}\right\}
\end{equation*}
and
\begin{equation*}
V^a:=\left\{f\in L^1([0, a^\dagger], \R)\ \big|\ f(a)=\int_0^{a}\varphi(\hat a)\dd\hat a\text{ for some }\varphi\in L^1([0, a^\dagger], \R)\right\}.
\end{equation*}
It can be shown that  $\mathcal  K^a$ is the generator of an eventually compact strongly continuous semigroup of bounded positive linear operators on $V^a$ \cite[Appendix A]{ando2022pseudospectral}. Thus its spectrum consists of only eigenvalues which are at most countable and either are isolated or they accumulate at $-\infty$. Moreover, it can be shown that $\gamma_0:=s(\mathcal  K^a)$ is actually an eigenvalue of $\mathcal  K^a$. 
Let then $\{\gamma_i\}_{i=0}^{\infty}$ be the eigenvalues of $\mathcal  K^a$ and let us arrange them as $\gamma_0> Re\gamma_1\ge Re\gamma_2\ge \cdots$\footnote{Note that the eigenvalues of $\mathcal K^a$ coincide with those of the operator in \cite[Theorem 2.2]{kang2020age}. See also \cite[Appendix A]{ando2022pseudospectral}.}.

\bigskip
As far as space-dependency is considered, let $\mathcal  J^x\colon L^2(\Omega,\R)\to L^2(\Omega,\R)$ be defined as
\begin{equation*}
\mathcal J^xg(x):=\int_{\Omega}J(x-y)g(y)\dd y.
\end{equation*}
Under the assumptions on $J$, $\mathcal J^x$  is compact and self-adjoint and the following result on its eigenvalues holds \cite{garcia2009principal}.
\begin{proposition}\label{teoremaRossi}
$\mathcal J^x$ has only real eigenvalues smaller than $1$, say $\{\theta_j\}_{j=0}^{\infty}$ ordered as $1>\theta_0>\theta_1\ge \theta_2\ge\cdots$ and they accumulate at $0$. Moreover, $\theta_0$ is the  unique eigenvalue associated to a positive eigenfunction $g_0\in L^2(\Omega,\R)$, it is simple and positive.
\end{proposition}

Now, if $\lambda\in \sigma(\mathcal B)$ with eigenfunction $\psi\in D(\mathcal B)$, i.e., 
\begin{equation*}
\mathcal B\psi=\lambda \psi,
\end{equation*}
then, under \cref{assumptionrates}, it is shown in \cite{kang2020age} that $\psi(a, x)=f(a)g(x)$ for $f$ an eigenfunction of $\mathcal K^a$ and $g$ an eigenfunction of $\mathcal J^x$.
Moreover, the following result holds \cite{kang2020age}.
\begin{theorem}\label{decompositioneig}
Let $\{\gamma_i\}_{i=0}^\infty$ and $\{\theta_j\}_{j=0}^\infty$ be the eigenvalues of $\mathcal K^a$ and $\mathcal J^x$, respectively.
Then:
\begin{enumerate}
\item \label{teo3i} $\sigma(\mathcal B)=\{\gamma_i+\theta_j-1\}_{i,j=0}^{\infty}$;
\item \label{teo3ii} $\lambda_0:=\gamma_0+\theta_0-1$ is the principal eigenvalue of $\mathcal B$, it has positive eigenfunction
\begin{equation*}\psi_0(a, x)=\left(\int_0^ae^{-\gamma_0 \hat a}e^{-\int_0^{\hat a} \mu_{1}(s)\dd s}\dd\hat a\right) g_0(x)\end{equation*}
and $s(\mathcal B)=\lambda_0$.
\end{enumerate}
\end{theorem}

\section{The numerical approach}\label{sezioneapproccio}
The aim of the numerical method described next is that of approximating (the rightmost) part of $\sigma_p(\mathcal B)$ trough the eigenvalues of a matrix discretizing $\mathcal B$.
We have in mind domains $\Omega\subset\R^n$ which are (or can be mapped to) rectangles and we propose to combine a spectral projection in space with a pseudospectral collocation in age.

\bigskip
As for the spectral projection in space, we restrict to the case $\Omega=(-l,l)$, $l>0$, for the sake of simplicity.\footnote{Extension to rectangular domains tipically rely on the tensorial product of univariate orthonormal bases \cite{canuto1988spectral}.}
Let $\{P_j\}_{j=0}^{\infty}$ be the Legendre orthonormal basis for $L^2(\Omega,\R)$, $M$ be a positive integer and $V_M^x:=\Span\left\{P_0,\dots ,P_{M}\right\}$. 
We define restriction and prolongation operators respectively as 
\begin{equation*}
\mathcal R_{M}^x\colon L^2(\Omega,\R)\to \R^{M+1}, \qquad \mathcal R_{M}^xg=(g_0; \dots; g_M)\footnote{We adopt the MATLAB-like notation according to which elements of a column vector are separated by ``;'' while elements of a row vector are separated by ``,''.},
\end{equation*}
 where $g_j:=\left \langle g, P_j \right \rangle_{L^2(\Omega, \R)}$ for all $j=0,\dots, M$ and 
\begin{equation*}
\mathcal P_{M}^x\colon \R^{M+1 }\to V_M^x, \qquad \mathcal P_{M}^xG=\sum_{j=0}^MP_jG_j,
\end{equation*}
where $G:=(G_0;\dots; G_M)$.
Observe that 
\begin{equation*}
\mathcal R_{M}^x \mathcal {P}_M^x=\mathcal I_{\R^{M+1}},\qquad
\mathcal P_M^x \mathcal R_M^x=\pi_M,
\end{equation*}
where $\pi_M:L^2(\Omega,\R)\to V_M^x$ is the projection operator $\pi_Mg=\sum_{j=0}^MP_jg_j$.

\bigskip
As for the pseudospectral collocation in age, let $N$ be a positive integer, $A_N:=\{a_1<\dots< a_N\}$ be the mesh of Chebyshev zeros in $[0, a^\dagger]$, $\Pi_N^a([0, a^\dagger],	\R)$ be the space of algebraic polynomials on $[0, a^\dagger]$ of degree at most $N$ and 
\begin{equation*}
V^a_N:=\{ f_N\in\Pi_N^a([0, a^\dagger],\R)\ |\ f_N(0)=0\}.
\end{equation*} 
We define restriction and prolongation operators respectively as
\begin{equation*}
\mathcal R_{N}^a\colon V^a\to \R^{N},\qquad \mathcal R_{N}^af=(f(a_1);\dots; f(a_N))
\end{equation*}
and 
\begin{equation*}
\mathcal P_{N}^a\colon \R^{N}\to V^a,\qquad \mathcal P_{N}^aF=\sum_{i=1}^N\ell_{N,i}F_i,
\end{equation*}
where $F:=(F_1;\dots; F_N)$ and $\{\ell_{N,i}\}_{i=0}^N$ is the Lagrange basis relevant to the nodes in $\{0\}\cup A_N$.
Observe that
\begin{equation*}
\mathcal R_{N}^a \mathcal {P}_{N}^a=\mathcal I_{\R^N},\qquad
\mathcal P_{N}^a \mathcal R_{N}^a=\mathcal L_N,
\end{equation*}
where $\mathcal L_N:V^a\to V^a_N$ is the Lagrange interpolation operator relevant to $\{0\}\cup A_N$.

\bigskip
A finite-dimensional approximation $\mathcal B_{N,M}\colon \R^{N(M+1)}\to\R^{N(M+1)}$ of $\mathcal B$ is obtained as
\begin{equation*}
\mathcal B_{N,M}:=\mathcal R_N^a \mathcal R_M^x \mathcal B \mathcal P_M^x\mathcal P_N^a,\footnote{Observe that in exact arithmetic this is equivalent to $\mathcal R_M^x\mathcal R_N^a \mathcal B \mathcal P_N^a\mathcal P_M^x$.}
\end{equation*}
which, according to \eqref{split}, reads 
\begin{equation}\label{bnm}
\mathcal B_{N,M}=d\mathcal R_N^a \mathcal R_M^x (\mathcal J^x-\mathcal I_V) \mathcal P_M^x\mathcal P_N^a+\mathcal R_N^a \mathcal R_M^x \mathcal K^a \mathcal P_M^x\mathcal P_N^a.  
\end{equation}

Note that, under \cref{assumptionrates}, \eqref{bnm} reduces to
\begin{equation}\label{bnm1}
\mathcal B_{N,M}=d(\mathcal J_M^x-\mathcal I_{\R^{M+1}})\otimes\mathcal I_{\R^N}+\mathcal I_{\R^{M+1}}\otimes\mathcal K_N^a
\end{equation}
for 
\begin{equation*}
\mathcal J_M^x\colon \R^{M+1}\to \R^{M+1},\qquad \mathcal J_M^x:=\mathcal R_M^x \mathcal J^x \mathcal P_M^x
\end{equation*} 
and
\begin{equation}\label{Kna}
\mathcal K_N^a\colon \R^{N}\to \R^{N},\qquad\mathcal K_N^a:=\mathcal R_N^a \mathcal K^a \mathcal P_N^a.
\end{equation} 
Thus, in this case, computing separately the eigenvalues of $\mathcal J_M^x$ and those of $\mathcal K_N^a$ represents an efficient alternative to compute directly the eigenvalues of \eqref{bnm}. Indeed, 
\begin{equation}\label{spettriequiv}
\sigma\left(\mathcal B_{N,M}\right)=\{\gamma+\theta-1\ |\ \theta\in\sigma(\mathcal J_M^x)\text{ and }\gamma\in \sigma(\mathcal K_{N}^a)\}
\end{equation}
follows from \cite[Theorem 4.4.5]{horn1991topics} and represents the numerical counterpart of \cref{decompositioneig} \ref{teo3i}.

\bigskip
In the following section we provide a rigorous convergence analysis of the eigenvalues of \eqref{bnm1} to those of $\mathcal B$. In fact, under \cref{assumptionrates} it is known that $\mathcal B$ as only point spectrum and that $\lambda_0=s(\mathcal B)$ (see \cref{decompositioneig} \ref{teo3ii}). This convergence analysis is not suitable for the general case in which $\sigma(\mathcal B)$ is not necessarily just point spectrum. Finally, even if a principal eigenvalue $\lambda_0=s(\mathcal B)$ is guaranteed to exist (see \cite{ducrot2022age}), studying the convergence properties of its approximation may require an ad-hoc treatment which is out of the scope of the current work. Nevertheless, also for this case, in \cref{numresults} we provide some encouraging numerical results as we comment therein.

\section{Convergence analysis}\label{sezdim}
The following analysis holds under \cref{assumptionrates} and is based on proving separately the convergence of the eigenvalues of $\mathcal J_M^x$ to those of $\mathcal J^x$ as $M\to \infty$ (section \ref{dissconv}) and the convergence of the eigenvalues of $\mathcal K_N^a$ to those of $\mathcal K^a$ as $N\to \infty$ (section \ref{convsenzadiff}). Then \eqref{spettriequiv} provides the final result on the convergence of  the eigenvalues of $\mathcal B_{N,M}$ to those of $\mathcal B$ as $N, M\to \infty$ (section \ref{5.3}). 

\subsection{Convergence in space}\label{dissconv}
The following main theorem is based on the compactness of $\mathcal J^x$ and on standard results reported, e.g., in \cite{chatelin2011spectral}.

\begin{theorem}\label{teospec}
Let $\theta\in\R$ be an isolated nonzero eigenvalue of $\mathcal J^x$ with finite algebraic multiplicity $m$ and ascent $l$ and let $\Delta$ be a neighborhood of $\theta$ such that $\theta$ is the sole eigenvalue of $\mathcal J^x$ in $\Delta$. Then there exists $\overline M$ such that, for $M\ge \overline M$, $\mathcal J_M^x$ has in $\Delta$ exactly $m$ eigenvalues $\theta_{M, j}$, $j=1,\dots, m$, counting their multiplicities. Moreover, 
\begin{align*}
\max_{j=1, \dots, m}|\theta_{M, j}-\theta|=O(\varepsilon_M^{1/l})
\end{align*}
for
\begin{equation*}
\varepsilon_M:=\|\pi_M^x\mathcal J^x-\mathcal J^x\|_{L^2(\Omega,\R)\leftarrow \mathcal M_{\theta}}\footnote{In general for a linear operator $L\colon X\to Y$ we denote $\|L\|_{Y\leftarrow X}:=\sup_{\|x\|_X=1}\|Lx\|_Y$.}
\end{equation*}
and $\mathcal M_{\theta}$ the generalized eigenspace of $\theta$. 
Finally, for any $j=1,\dots, m$ and for any eigenvector $g_{M, j}$ of $\pi_M^x\mathcal J^x$ relevant to $\theta_{M,j}$ such that $\|g_{M,j}\|=1$, we have
\begin{equation*}
\dist(g_{M,j},\  \ker(\theta \mathcal I_{L^2(\Omega,\R)}-\mathcal J^x))=O(\varepsilon_M^{1/l}),
\end{equation*}
where $\dist$ is the distance in the space $L^2(\Omega, \R)$ between an element and a subspace.
\end{theorem}
\begin{proof}
We first observe that $\mathcal J_M^x$ has the same nonzero eigenvalues with the same algebraic and geometric multiplicities of $\pi_M^x\mathcal J^x$ \cite[Proposition 4.1]{breda2012approximation}.
As for the latter, the compactness of $\mathcal J_x$, the pointwise convergence of $\pi_M^x$ to $\mathcal I_{L^2(\Omega, \R)}$ \cite[Theorem 5.9]{brezis2011functional} and \cite[Theorem 4.5]{chatelin2011spectral} give
$\|\pi_M^x\mathcal J^x-\mathcal J^x\|_{L^2(\Omega, \R)\leftarrow L^2(\Omega, \R)}\to 0\text{ as }M\to \infty$.
The thesis now follows from \cite[Theorem 6.7]{chatelin2011spectral}.
\end{proof}
The analysis is now completed with a bound on the error on the eigenvalues depending on the regularity of $\mathcal J^x$.
\begin{assumption} \label{Assumptions 5}
$\mathcal J^x\in H^k(\Omega, \R)$ for some integer $k>0$. 
\end{assumption}
\begin{corollary}\label{corollariospec}
Let $\theta$ and $\theta_{M,j}$, $j=1,\dots, m$, be as in \cref{teospec}. Under \cref{Assumptions 5}
\begin{align*}
\max_{j=1, \dots, m}|\theta_{M, j}-\theta|=O\left(\cfrac{1}{M^{k/l}}\right).
\end{align*}
\end{corollary}
\begin{proof}
\cref{Assumptions 5} implies that $\mathcal M_{\theta}\subset H^k(\Omega, \R)$.
Then, from \cref{teospec} and \cite[p. 309, Formula (9.7.11)]{canuto1988spectral}, for every $g\in \mathcal M_{\theta}$ there exists $C>0$  such that 
\begin{equation*}
\|(\pi_M\mathcal J^x-\mathcal J^x)g\|_{L^2(\Omega, \R)}
\le C\cdot  \frac{1}{M^k}\cdot \|\mathcal J^xg\|_{H^k(\Omega, \R)}=C\cdot|\theta|\cdot  \frac{1}{M^k}\cdot \|g\|_{H^k(\Omega, \R)}
\end{equation*}
holds.
\end{proof}

\subsection{Convergence in age}\label{convsenzadiff}
The following analysis is based on \cite[Section 4]{breda2020collocation} to which we refer for the general results. Here we only need to construct the characteristic equation associated to $\mathcal K^a$, its discrete counterpart associated to $\mathcal K_N^a$ and provide a bound on the relevant collocation error depending on one among the following regularity conditions on $\mu_1$.  
\begin{assumption}\label{Assumptions 4}
\hspace{0mm}
\begin{enumerate}
\item \label{ass4i} $\mu_{1}\in C^s([0, a^\dagger],\R)$ for some integer $s\ge 0$;
\item \label{ass4ii} $\mu_{1}\in C^{\infty}([0, a^\dagger],\R)$;
\item \label{ass4iii} $\mu_{1}$ is real analytic.
\end{enumerate}
\end{assumption}
\bigskip
Let $\gamma\in \C$ and $f\in D(\mathcal K^a)$ be such that
\begin{equation*}
\mathcal K^af=\gamma f. 
\end{equation*}
By setting $\eta:=f'$ we get
\begin{equation}\label{eqver}
-\eta(a)+\int_0^{a^\dagger}\beta(\hat a)\Pi_0(\hat a)\eta(\hat a)\dd\hat a-\int_0^{a}\mu_{1}(\hat a)\eta(\hat a)\dd\hat a=\gamma \int_0^a\eta(\hat a)\dd \hat a.
\end{equation}
Thus $\gamma$ is an eigenvalue for $\mathcal K^a$ if and only if 
\begin{equation*}
1-\hat{\mathcal K}^a(\gamma)=0
\end{equation*}
for 
\begin{equation}\label{char1}
\hat{\mathcal K}^a(\gamma):=\int_0^{a^\dagger}\beta(\hat a)\Pi(\hat a)e^{-\int_0^{\hat a} \mu_{1}(s)\dd s-\gamma\hat a}\dd\hat a.
\end{equation}
Note that \eqref{eqver} can be seen in the framework of
\begin{equation}\label{qsol}
q+\mathcal S\left(\gamma +\mu_{1}\right)q=\alpha
\end{equation}
with general $\alpha,\gamma\in \C$, which admits a unique solution $q=q(\cdot, \gamma, \alpha)$ thanks to standard results on Volterra integral equations \cite[Theorem 9.3.6]{gripenberg1990volterra}. In particular $\mathcal I_{L^1([0, a^\dagger],\R)}+\mathcal S(\gamma+\mu_{1})\colon L^1([0, a^\dagger],\R)\to L^1([0, a^\dagger],\R)$ is invertible with bounded inverse thanks to \cref{Assumption 1} \ref{ass1i}.
In fact, $q$ is the derivative $\eta$ of an eigenfunction $f$ of $\mathcal K^a$ when $\gamma\in\sigma(\mathcal K^a)$ (and $\alpha=\int_0^{a^\dagger}\beta(\hat a)\Pi_0(\hat a)\eta(\hat a)\dd\hat a$ follows from the domain of $\mathcal K^a$).

Similarly, let $\gamma\in \C$ and $f_N\in V_N^a$ be such that
\begin{equation*}
\mathcal K_N^aF=\gamma F
\end{equation*}
for $F:=(f_N(a_1);\dots; f_N(a_N))$.
By setting $\eta_N:=f_N'$ we get, for $i=1,\dots, N$,
\begin{equation*}
-\eta_N(a_i)+\int_0^{a^\dagger}\beta(\hat a)\Pi_0(\hat a)\eta_{N}(\hat a)\dd\hat a-\int_0^{a_i}\mu_{1}(\hat a) \eta_N(\hat a)\dd\hat a=\gamma \int_0^{a_i}\eta_N(\hat a)\dd \hat a.
\end{equation*}
If $q_N=q_{N}(\cdot; \gamma, \alpha)$ is the polynomial collocating \eqref{qsol} 
at $A_N$, we get that $\gamma$ is an eigenvalue of $\mathcal K_N^a$ if and only if 
\begin{equation*}
1-\hat{\mathcal K}_N^a(\gamma)=0
\end{equation*}
for 
\begin{equation}\label{char2}
\hat{\mathcal K}_N^a(\gamma):=\int_0^{a^\dagger}\beta(\hat a)\Pi_0(\hat a)q_N(\hat a; \gamma, 1)\dd\hat a.
\end{equation}
Note that $q_N(\cdot; \gamma, 1)=\eta_N(\cdot; \gamma, 1)$ when $\gamma\in\sigma(\mathcal K_N^a)$.

\bigskip
Now we show that $q_N$ exists and is unique and give an upper bound for $\|q_N(\cdot; \gamma, \alpha)-q(\cdot; \gamma, \alpha)\|_{L^1([0, a^\dagger],\R)}$.
Since $q_N$ is a polynomial of degree $N-1$, it can be reconstructed exactly at $N$ nodes trough the Lagrange interpolation operator ${\tilde {\mathcal L}_{N-1}}$ relevant to $A_N$. Then
\begin{equation}\label{feq2}
q_N+{\tilde {\mathcal L}_{N-1}}\left[\mathcal S(\gamma +\mu_{1})q_N \right]=\alpha
\end{equation}
follows straightforwardly from \eqref{qsol} and
\begin{equation}\label{uniqueness1}
e_N+{\tilde {\mathcal L}_{N-1}}\mathcal S(\gamma +\mu_{1})e_N= r_N 
\end{equation}
follows by setting $e_N:=q_N-q$ and $r_N:=(\mathcal I_{L^1([0, a^\dagger],\R)}-{\tilde {\mathcal L}_{N-1}})\mathcal S(\gamma+\mu_{1})q$ (we avoid to write the dependence of $e_N$ and $r_N$ on $\alpha$ and $\gamma$, unless explicitly required).

\begin{proposition}\label{teoremaunicità1dim}
Under \cref{Assumption 1} \ref{ass1ii} there exists a positive integer $\overline N$ such that, for $N\ge \overline N$, \eqref{uniqueness1} has a unique solution $e_N$ and 
\begin{equation*}\label{ordconv1}
\|e_N\|_{L^1([0, a^\dagger],\R)} \le 2|\alpha| \varepsilon_N\left\| \left(\mathcal I_{L^1([0, a^\dagger],\R)}+\mathcal S(\gamma+\mu_{1}))\right)^{-1}\right\|_{L^1([0, a^\dagger],\R)\leftarrow L^1([0, a^\dagger],\R)},
\end{equation*}
where
\begin{equation*}
\varepsilon_N=\begin{cases} O(N^{-(s+1)}\log N)&\text{under \cref{Assumptions 4} \ref{ass4i}};\\
O(N^{-(r+1)}\log N)&\text{for every integer }r\ge 1\text{ under \cref{Assumptions 4} \ref{ass4ii}};\\
O(p^{-N}\log N)&\text{for some constant } p>1\text{ under \cref{Assumptions 4} \ref{ass4iii}}.
\end{cases}
\end{equation*}
\end{proposition}
\begin{proof}
By observing that $\Range(\mathcal S )\subset C([0, a^\dagger],\R)$ we have that 
\begin{equation*}
\|\mathcal I_{C([0, a^\dagger],\R)}-{\tilde {\mathcal L}_{N-1}}\|_{L^1([0, a^\dagger],\R)\leftarrow C([0, a^\dagger],\R)} \to 0\text{ as }N\to \infty
\end{equation*}
 thanks to \cite[Theorem Ia]{erdos1937interpolation}, thus we can apply the Banach perturbation Lemma \cite[Theorem 10.1]{kress1989linear} to get 
\begin{align*}\label{ordconv2}
\|e_N\|_{L^1([0, a^\dagger],\R)} \le &2\left\|\left(\mathcal I_{L^1([0, a^\dagger],\R)}+\mathcal S(\gamma+\mu_{1}))\right)^{-1}\right\|_{L^1([0, a^\dagger],\R)\leftarrow L^1([0, a^\dagger],\R)}\\[2mm]&\cdot \|r_N\|_{L^1([0, a^\dagger],\R)}.
\end{align*}
Then
\begin{equation*}
\|r_N\|_{L^1([0, a^\dagger],\R)}\le a^\dagger \|r_N\|_{\infty}\le
\left(1+\Lambda_{N-1}\right)E_{N-1}\left(\mathcal S(\gamma+\mu_{1})q\right)
\end{equation*}
follows from standard results on uniform approximation for $\Lambda_{N-1}$ the Lebesgue constant relevant to $A_N$ and $E_{N-1}\left(f\right)$ the best  uniform approximation error of a continuous function $f$ on the space of polynomials of degree at most $N-1$. The final bound follows since $q$ depends linearly on $\alpha$, the choice of $A_N$ guarantees $\Lambda_{N-1}=O(\log N)$ and estimates of $E_{N-1}\left(\mathcal S(\gamma+\mu_{1})q\right)$ can be obtained by applying Jackson's type theorems \cite[Section 1.1.2]{rivlin1981introduction}.
\end{proof}

\bigskip
The final result follows by comparing $\hat{\mathcal K}^a$ in \eqref{char1} with $\hat{\mathcal K}^a_N$ in \eqref{char2}.
\begin{theorem}\label{theorem5.4}
Let $\gamma$ be a non zero eigenvalue of $\mathcal K^a$ with finite algebraic multiplicity $m$ and let $\Delta$ be a neighborhood of $\gamma$ such that $\gamma$ is the sole eigenvalue of $\mathcal K^a$ in $\Delta$. Under \cref{Assumption 1} there exists $\overline N\in \N$ such that, for $N\ge\overline N$, $\mathcal K_N^a$ has in $\Delta$ exactly $m$ eigenvalues $\gamma _{N,i}$, $i=1,\dots,m$, counting their multiplicities. Moreover,
\begin{equation*}
\max_{i=1,\dots, m} |\gamma_{N,i}-\gamma|\le \rho(N)
\end{equation*}
with
\begin{equation*}
\rho(N)=\begin{cases} O((N^{-(s+1)}\log N)^{1/m}) &\text{under \cref{Assumptions 4} \ref{ass4i}};\\
O((N^{-(r+1)}\log N)^{1/m})&\text{for every integer }r\ge 1\\&\qquad\text{under \cref{Assumptions 4} \ref{ass4ii}};\\
O((p^{-N}\log N)^{1/m})&\text{for some constant } p>1\\ &\qquad\text{under \cref{Assumptions 4} \ref{ass4iii}}.
\end{cases}
\end{equation*}
\end{theorem}
\begin{proof}
The thesis follows from
\begin{gather*}
|\hat{\mathcal K}^a_N(\gamma)-\hat{\mathcal K}^a(\gamma)|\le \int_0^{a^\dagger} |\beta(\hat a)\Pi_0(\hat a)|\cdot\left | q_N(\hat a; \gamma, 1)-q(\hat a; \gamma,1)\right|\dd\hat a\le\\
\le a^\dagger \|\beta\Pi_0\|_{L^\infty([0, a^\dagger],\R)}\left\|e_N(\cdot; \gamma, 1)\right\|_{L^1([0, a^\dagger],\R)}
\end{gather*}
by using Proposition \ref{teoremaunicità1dim} and (eventually) by applying Rouché's Theorem as in \cite[Section 4]{breda2020collocation}.
\end{proof}
\begin{remark}\label{remarkPCW}
The convergence result of \cref{theorem5.4} is preserved in the case one uses a straightforward piecewise collocation approach in presence of possible discontinuities of $\beta$ and $\mu_{1}$ (or of their derivatives). Note that the MATLAB demos available at \url{http://cdlab.uniud.it/software} that we use to make all the tests in sections \ref{numresults} and \ref{neumann} implement this piecewise alternative. 
\end{remark}

\subsection{Final convergence result}\label{5.3}

Summarizing the above results, the main theorem follows as a corollary.
\begin{theorem}\label{teoremaconvergenzaeig}
Let \cref{Assumptions 5} hold.
Let $\lambda\in\sigma(\mathcal B)$, $\gamma\in \sigma(\mathcal K^a)$ and $\theta\in \sigma(\mathcal J^x)$ be such that $\lambda=\gamma+\theta-1$. Suppose that $\gamma$ is a nonzero eigenvalue of $\mathcal K^a$ with finite algebraic multiplicity $m_\gamma$ and let $\Delta_\gamma$ be a neighborhood of $\gamma$ such that $\gamma$ is the sole eigenvalue of $\mathcal K^a$ in $\Delta_\gamma$.
Moreover, suppose that  $\theta\in\R$ is an isolated nonzero eigenvalue of $\mathcal J^x$ with finite algebraic multiplicity $m_\theta$, ascent $l_\theta$ and $\Delta_\theta$ is a neighborhood of $\theta$ such that $\theta$ is the sole eigenvalue of $\mathcal J^x$ in $\Delta_\theta$. 

Then there exist $\overline N,\overline M\in \N$ such that, for $N\ge \overline N$ and $M\ge \overline M$, ${\mathcal K}_N$ has in $\Delta_\gamma$ exactly $m_\gamma$ eigenvalues $\gamma_{N,i}$, $i=1,\dots, m_\gamma$, counting their multiplicities, and $\mathcal J_M^x$ has in $\Delta_\theta$ exactly $m_\theta$ eigenvalues $\theta_{M, j}$, $j=1,\dots, m_\theta$, counting their multiplicities, and
\begin{equation*}
\max_{\substack{i=1,\dots, m_\gamma\\ j=1,\dots, m_\theta}} |\lambda-(\gamma_{N,i}+\theta_{M,j}-1)|\le \rho_\gamma(N)+\rho_\theta(M)
\end{equation*}
with
\begin{gather*}
\rho_\gamma(N)=\begin{cases} O((N^{-(s+1)}\log N)^{1/m_\gamma}) &\text{under \cref{Assumptions 4} \ref{ass4i}};\\
O((N^{-(r+1)}\log N)^{1/m_\gamma})&\text{for every integer }r\ge 1\\&\qquad\text{ under \cref{Assumptions 4} \ref{ass4ii}};\\
O((p^{-N}\log N)^{1/m_\gamma})&\text{for some constant } p>1\\&\qquad\text{under \cref{Assumptions 4} \ref{ass4iii}};
\end{cases}\\
\end{gather*} 
and $\rho_\theta(M)=O\left(M^{-\frac{k}{l_\theta}}\right)$.
\end{theorem}

\section{Implementation}\label{secimplementation}
We now describe how to explicitly construct the entries of the matrices representing the discretized operators obtained in \cref{sezioneapproccio}.

\medskip
Let $(i;j):=(i-1)(M+1)+j$ for $i=1,\dots, N$ and $j=0,\dots, M+1$.
Thanks to the cardinal property of the Lagrange polynomials $\ell_{N,h}(a_i)=\delta_{ih}$, $i,h=0,\dots, N$, and to the orthonormality of $\{P_j\}_{j=0}^\infty$, the action of $\mathcal B_{N,M}$ on $\Psi=(\Psi_0;\dots; \Psi_{N(M+1)})\in\R^{N(M+1)}$ reads
\begin{gather*}
\left[\mathcal B_{N,M}\Psi\right]_{(i; j)}=
d\sum_{k=0}^M \left[ \int_{\Omega}(\mathcal J^xP_k-P_k)(x)P_j(x)\dd x\right]\Psi_{(i;k)}-
\sum_{h=1}^N\ell_{N,h}'(a_i)\Psi_{(h; j)}\\
+\sum_{h=1}^N\sum_{k=0}^M \left[ \int_\Omega P_k(x) \left(\int_0^{a^\dagger}\beta(\hat a, x)\Pi_0(\hat a)\ell_{N,h}'(\hat a)\dd\hat a \right)P_j(x)\dd x\right]\Psi_{(h; k)}\\-
\sum_{h=1}^N\sum_{k=0}^M \left[\int_\Omega P_k(x)\left(\int_0^{a_i}\mu_{1}(\hat a, x)\ell_{N,h}'(\hat a)\dd\hat a\right) P_j(x) \dd x\right]\Psi_{(h; k)}.
\end{gather*}
In particular, the matrix expression for $\mathcal B_{N,M}$ reads
\begin{equation*}
\mathcal B_{N,M}=d(\mathcal J_M^x-\mathcal I_{\R^{M+1}})\otimes \mathcal I_{\R^{N}}-\mathcal I_{\R^{M+1}}\otimes \mathcal D_N^a +\mathcal H_{N,M}-\mathcal W_{N,M}
\end{equation*}
where, for $i,h=1,\dots, N$ and $j,k=0,\dots, M$,
\begin{align}
&\label{espexplicitcomb1}(\mathcal J_{M}^x)_{j,k}:=\int_\Omega\mathcal J^xP_k(x) P_j(x)\dd x,\\
&\notag(\mathcal D_{N}^a)_{i,h}:=\ell'_{N,h}(a_i),\\
&\label{espexplicitcomb2}(\mathcal H_{N,M})_{(i;j), (h, k)}:=\int_\Omega P_k(x) \left(\int_0^{a^\dagger}\beta(\hat a, x)\Pi_0(\hat a)\ell_{N, h}'(\hat a)\dd\hat a \right)P_j(x)\dd x,\\
&\label{espexplicitcomb3}(\mathcal W_{N,M})_{(i;j), (h, k)}:=\int_\Omega P_k(x)\left(\int_0^{a_i}\mu_{1}(\hat a, x)\ell_{N, h}'(\hat a)\dd\hat a\right) P_j(x) \dd x.
\end{align}
Observe that $\mathcal J_M^x$ is symmetric since $\mathcal J^x$ is self-adjoint.
Moreover, under \cref{assumptionrates}, one can write $\mathcal H_{N,M}=\mathcal I_{\R^{M+1}}\otimes\mathcal H_{N}^a$ and  $\mathcal W_{N,M}=\mathcal I_{\R^{M+1}}\otimes\mathcal W_{N}^a$, where, for $i,h=1,\dots, N$,
\begin{align*}
&(\mathcal H_N^a)_{i, h}:=\int_0^{a^\dagger}\beta(a)\Pi_0(a)\ell_{N,h}'(a)\dd a,\\
&(\mathcal W_N^a)_{i,h}:=\int_0^{a_i}\mu_{1}(a)\ell_{N,h}'(a)\dd a.
\end{align*}
Clearly, $\mathcal K_N^a$ in \eqref{Kna} reads
\begin{equation*}
\mathcal K_N^a=-\mathcal D_N^a +\mathcal H_N^a-\mathcal W_N^a.
\end{equation*}

Finally, if the integrals in \eqref{espexplicitcomb1}, \eqref{espexplicitcomb2} and \eqref{espexplicitcomb3}  (or the following ones) can not be computed analytically, we need to approximate them numerically. In this regards, we use the Gauss-Legendre quadrature for the integrals in space, the Clenshaw-Curtis quadrature for the integrals in age in $[0, a^\dagger]$ and the $i$-th row of the inverse of the differentiation matrix $\mathcal D_N^a$ for the integrals in $[0, a_i]$, $i=1,\dots, N$ \cite{diekmann2020pseudospectral}. Of course the relevant quadrature errors must be suitably taken into account in bounding the final error on the eigenvalues. The results in \cite{trefethen2008gauss} ensure that \cref{teoremaconvergenzaeig} still holds under regularity assumptions on the model coefficients. We also remark that by using Gauss-Legendre quadrature, the spectral discretization in space becomes equivalent to collocating at the Legendre-Gauss zeros \cite[Section 3.2.5]{canuto1988spectral}.

\section{Numerical results}\label{numresults}
The following test cases concern model \eqref{model}. In particular the first one confirms \cref{teoremaconvergenzaeig} under \cref{assumptionrates}. The second one experimentally shows that similar results are preserved even without \cref{assumptionrates} but still under the conditions of \cite{ducrot2022age}. Both concern $\Omega=(-l, l)$ for $l>0$. In the final third case, instead, we consider a circular domain (that indeed can be mapped into a rectangle as assumed at the beginning of \cref{sezioneapproccio}).

\medskip
\paragraph{Example 1}\label{ex1par}
Inspired by \cite{breda2007stability, kao2010random, kuniya2018global}, we consider 
$a^\dagger=1$,
\begin{equation*}
\beta(a, x)=\beta(a)=8(1-\ln R)\chi_{\left[\frac{1}{2},\ a^\dagger\right]}(a),\quad  \mu_0(a)=\frac{1}{a^\dagger-a}, \quad \mu_1(a, x)\equiv0
\end{equation*} 
and
\begin{equation*}
J(x)=\begin{cases}k e^{-\frac{1}{4-x^2}}&\text{if }x\in (-2, 2)\\ 0 &\text{otherwise}\end{cases}
\end{equation*} 
with $R>0$ and $k=\left(\int_{-2}^2e^{-\frac{1}{4-x^2}}\dd x\right)^{-1}$.
\begin{figure}
\centering
\includegraphics[width=.97\textwidth]{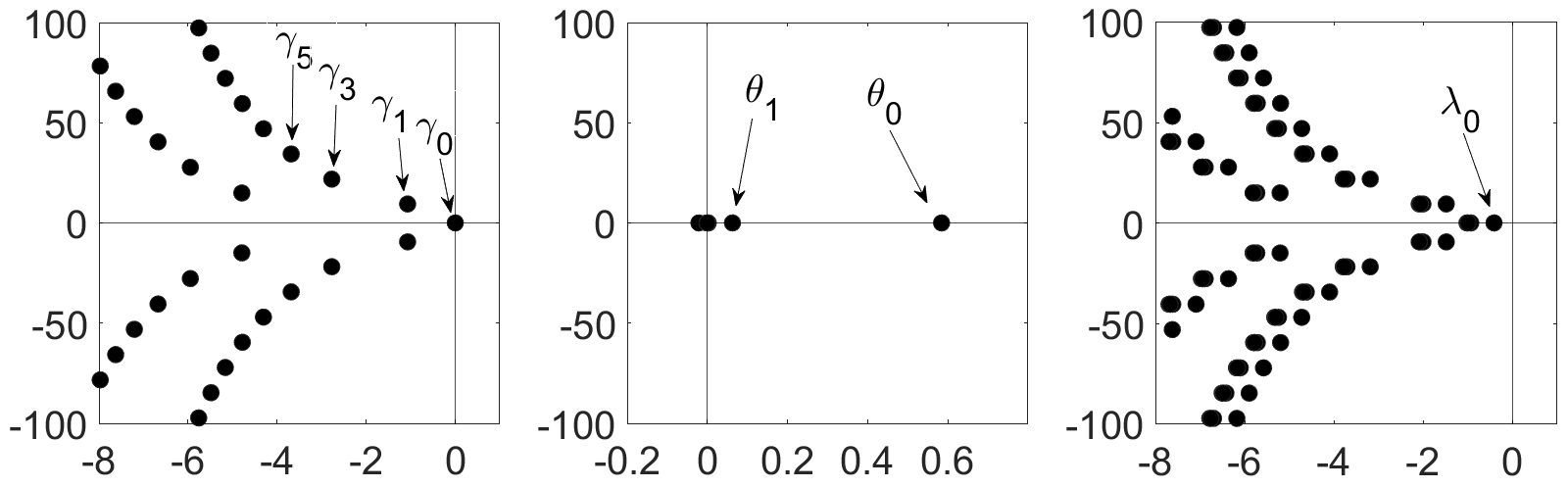}
\caption{\nameref{ex1par}: computed spectra in the complex plane with $l=1$ and $d=1$ for $\mathcal K^a$ (left), $J^x$ (center) and $\mathcal B$ (right) with $N=M=100$.}\label{primafig}

\medskip
\includegraphics[width=1.\textwidth]{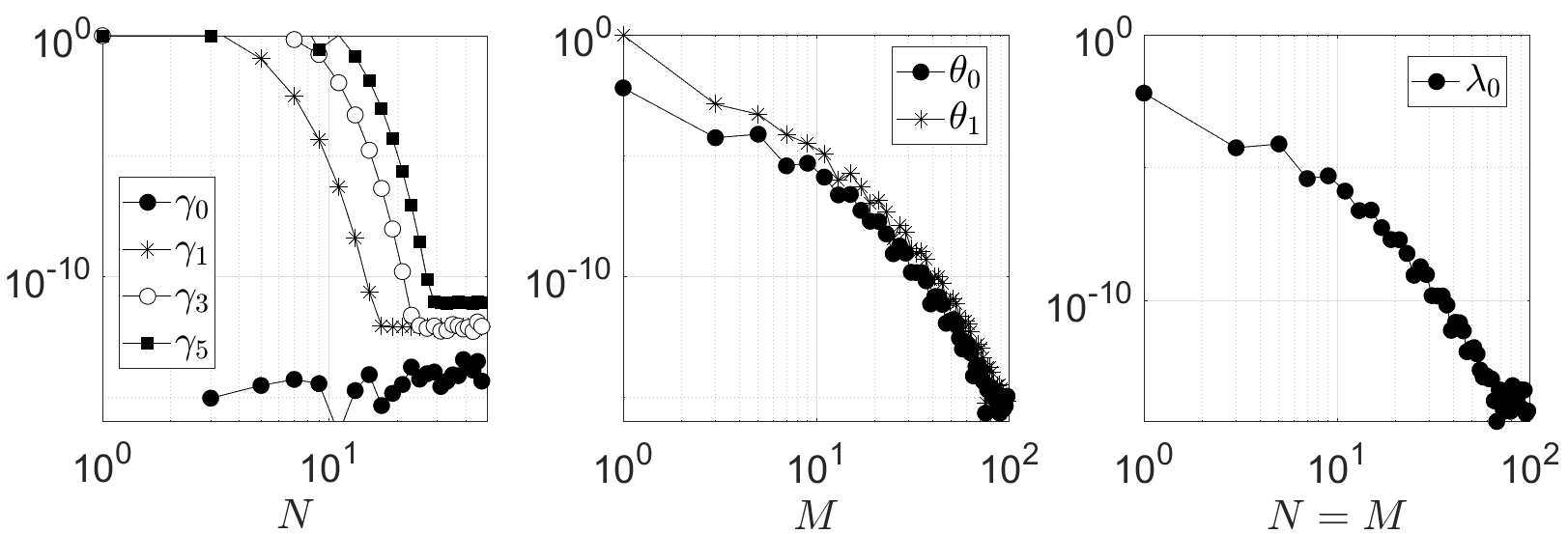}
\caption{\nameref{ex1par}: errors in computing $\gamma_0,\gamma_1, \gamma_3, \gamma_5$ (left), $\theta_0, \theta_1$ (center) and $\lambda_0$ (right) with $l=1$ and $d=1$ for increasing values of $N$ and $M$. The reference values $\gamma_1\approx -1.070249248997922 + 9.438599419667414i$, $\gamma_3\approx -2.772533303698945 +21.842287295644464i$, $\gamma_5\approx -3.683984412603065 +34.410428157635607i$ are obtained with $N=1000$; the reference values $\theta_0\approx 0.584294111974084$ and $\theta_1\approx 0.062556017866521$ are obtained with $M=1000$.}\label{terzafig}
\end{figure}
With these choices \cref{assumptionrates} holds, thus we compute the eigenvalues of $\mathcal K_N^a$ and $\mathcal J_M^x$ separately. We use $N=M$.
Moreover, $\Pi_0(a)=1-a$ and, for $R=1$, $\gamma_0=0$.\footnote{This can be derived as in \cite{breda2007stability} by observing that the characteristic equation for $\mathcal K^a$ reads 
\begin{equation*}
1=\int_{\frac{1}{2}}^1 8(1-\log R)(1-a)e^{-\lambda a}\dd a.\end{equation*}} 
Finally, since $\beta$ is piecewise defined, we resort to the piecewise approach mentioned in \cref{remarkPCW}.

\begin{figure}
\centering
\includegraphics[width=1.\textwidth]{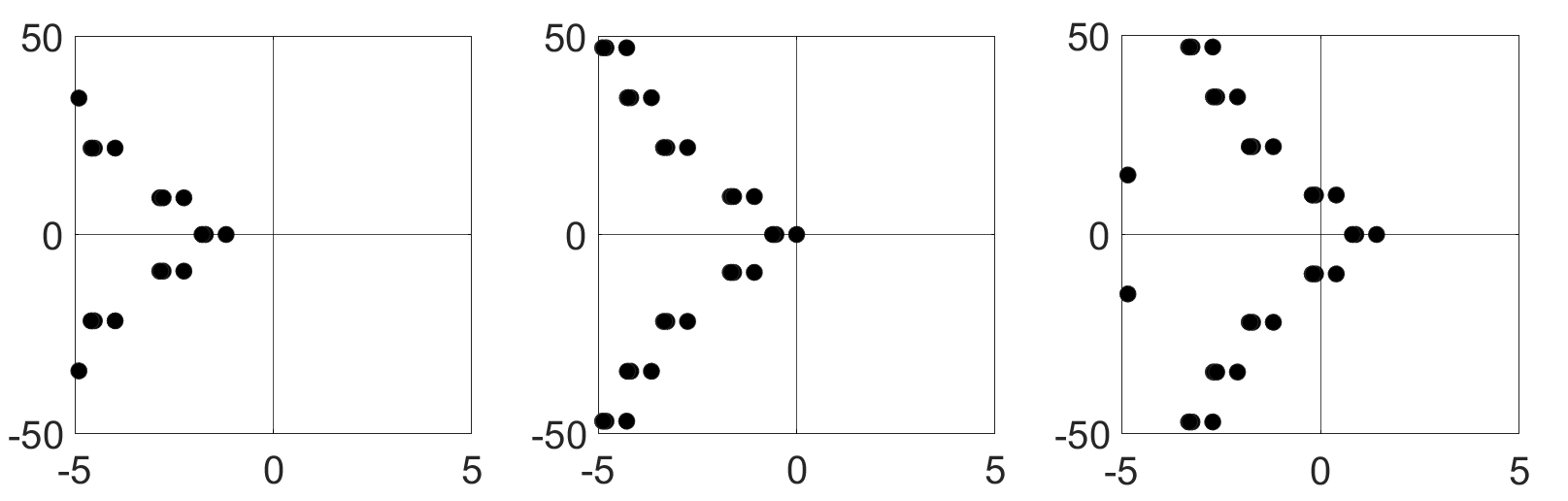}
\caption{\nameref{ex1par}: computed spectra of $\mathcal B$ in the complex plane with $l=1$, $d=1$ and $R=1.5$ (left), $R\approx 0.7277628676660066$ (center) and $R=0.1$ (right) with $N=M=100$.}\label{quintafig}
\end{figure}
Computed spectra are shown in \cref{primafig} for $l=1$ and $d=1$. For the same values, the convergence behavior is illustrated in \cref{terzafig}, which confirms \cref{corollariospec} and \cref{theorem5.4} showing also how the error constants are affected by $|\gamma|$ and $|\theta|$.

\medskip
\begin{figure}
\centering
\includegraphics[width=0.45\textwidth]{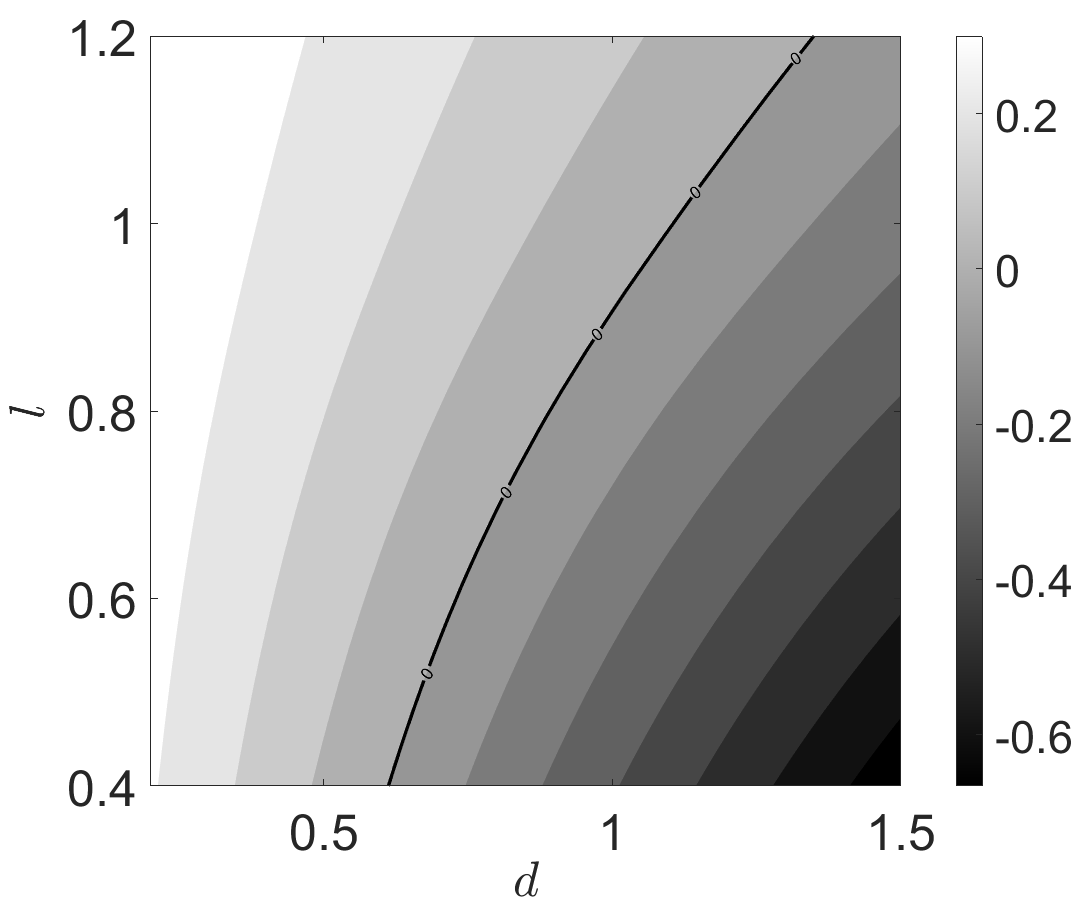}
\caption{\nameref{ex1par}: level curves for $\lambda_0$ as a function of $(d,l)$ and stability boundary (thick line) with $R=0.7$ and $N=M=30$.}\label{biforcazione1}
\end{figure}

With the proposed technique one can easily perform a bifurcation analysis of the null equilibrium with respect to $R$, \cref{quintafig} with $l=1$ and $d=1$, finding a transcritical bifurcation at $R\approx 0.7277628676660066$ (correspondingly $\gamma_0\approx 0.415705888025916$). Level curves of $\lambda_0$ as a function of $(d,l)$ are reported in \cref{biforcazione1}. The principal eigenvalue is a decreasing function of $d$ and an increasing function of $l$, confirming the results in \cite{ducrot2022age}.   

\medskip
\paragraph{Example 2}\label{ex2par}
We consider $a^\dagger=1$, 
\begin{equation*}
\beta(a, x)=\beta_0\chi_{\left[\frac{1}{2},\ a^\dagger\right]}(a)(2-x^2),\quad  \mu_0(a)=\frac{1}{a^\dagger-a}, \quad \mu_1(a, x)=\frac{1}{1+x^2}
\end{equation*}
and 
\begin{equation*}
J(x)=\pi^{-\frac{1}{2}}e^{-x^2}
\end{equation*}
with $\beta_0>0$.

Differently from \cref{ex1par}, $\beta$ and $\mu_1$ depend also on $x$.
Yet, these choices satisfy the conditions of \cite{ducrot2022age}.
Again, we use the piecewise approach with $N=M$.

\Cref{sestafig} shows the computed spectrum of $\mathcal B$ (left) and the convergence diagram for the first two rightmost eigenvalues $\lambda_0$ and $\lambda_1$ (right) with $\beta_0=8$ and $d=l=1$.
\cref{settimafig} and \cref{settimafig1} show the computed eigenfunctions of $\mathcal B$ and $\mathcal A$, respectively, for $\lambda_0$ and $\lambda_1$.
In \cref{settimafig1} (left) we observe the presence of a positive eigenfunction with a simple dominant real eigenvalue, i.e., $\lambda_0$ , whereas in \cref{settimafig1} (right) the eigenfunction of $\lambda_1$ changes sign. Note that in \cref{settimafig} (left) the corresponding eigenfunction of $\mathcal B$ is increasing and vanishes only for $a=0$ as it is relevant to the age-integrated state, whereas in \cref{settimafig1} (right) the eigenfunction of $\lambda_1$ is not monotone.  
Accordingly to the theory of \cite{ducrot2022age}, this suggests that the proposed approach correctly approximates the principal eigenvalue $\lambda_0$.
Eventually, \cref{biforcazione2} shows the level curves of $\lambda_0$ as a function of $(d, l)$. As in \nameref{ex1par}, $\lambda_0$ is a decreasing function of $d$ and an increasing function of $l$, confirming again the results in \cite{ducrot2022age}.
\begin{figure}
\centering
\includegraphics[width=0.73\textwidth]{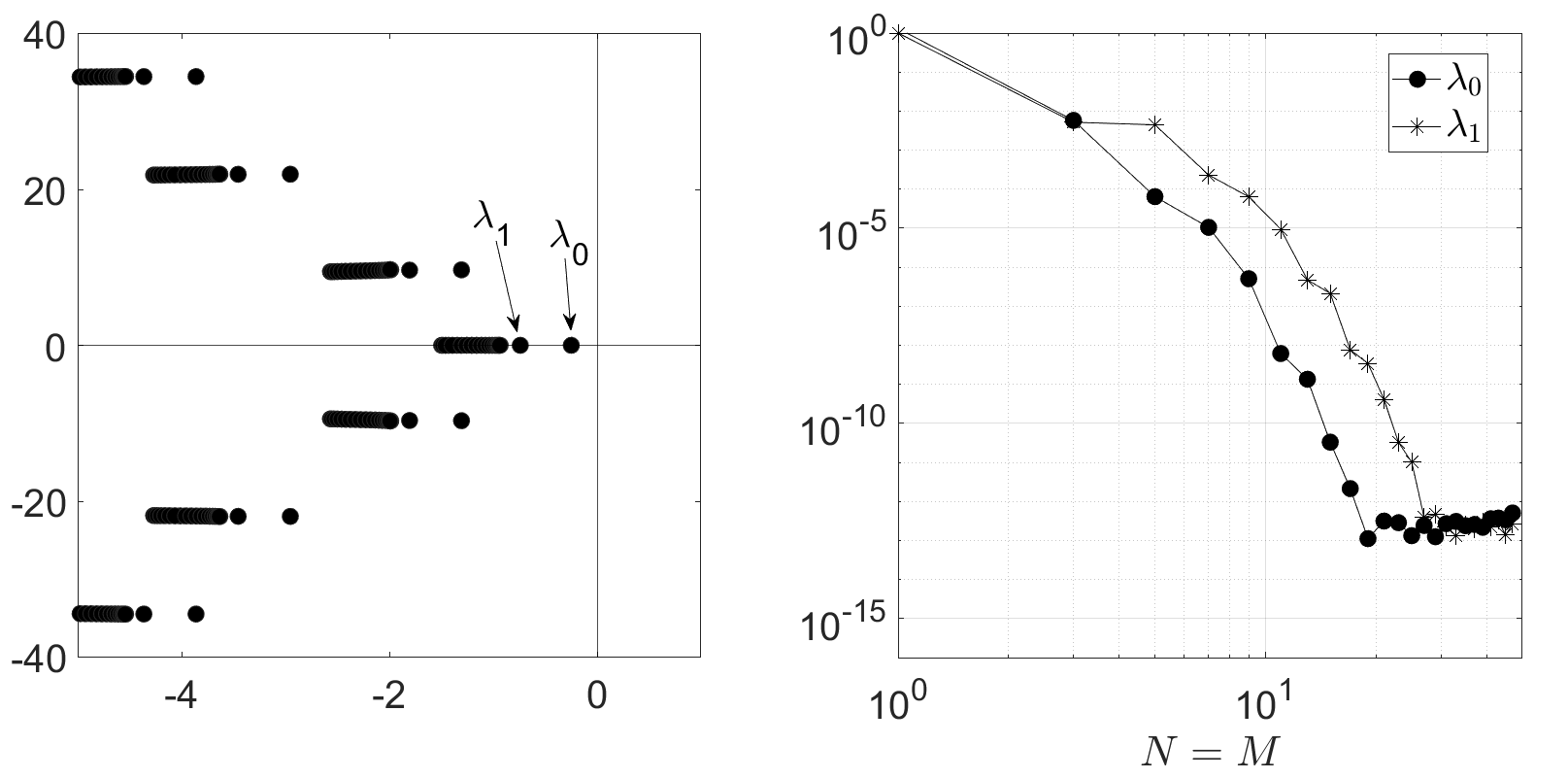}
\caption{\nameref{ex2par}: spectrum of $\mathcal B$ in the complex plane computed with $N=M=50$ (left) and convergence diagram for the first two rightmost eigenvalues $\lambda_0$ and $\lambda_1$ (right) with $\beta_0=8$ and $d=l=1$. The reference values $\lambda_0\approx-0.248872934970194$ and $\lambda_1\approx-0.739612643296491$ are obtained with $N=M=100$.}\label{sestafig}
\includegraphics[width=.9\textwidth]{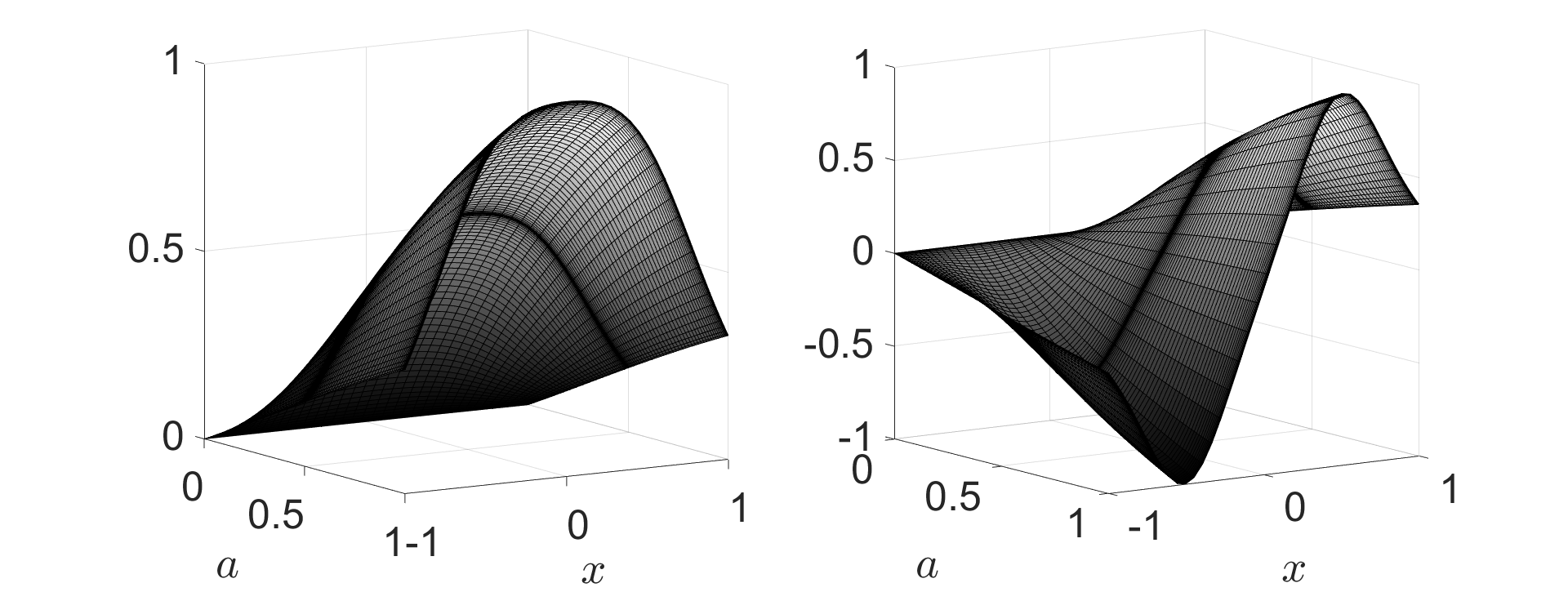}
\caption{\nameref{ex2par}: computed eigenfunctions of $\mathcal B$ relevant to $\lambda_0$ (left) and $\lambda_1$ (right) with $\beta_0=8$, $d=l=1$ and $N=M=50$.}\label{settimafig}
\includegraphics[width=.9\textwidth]{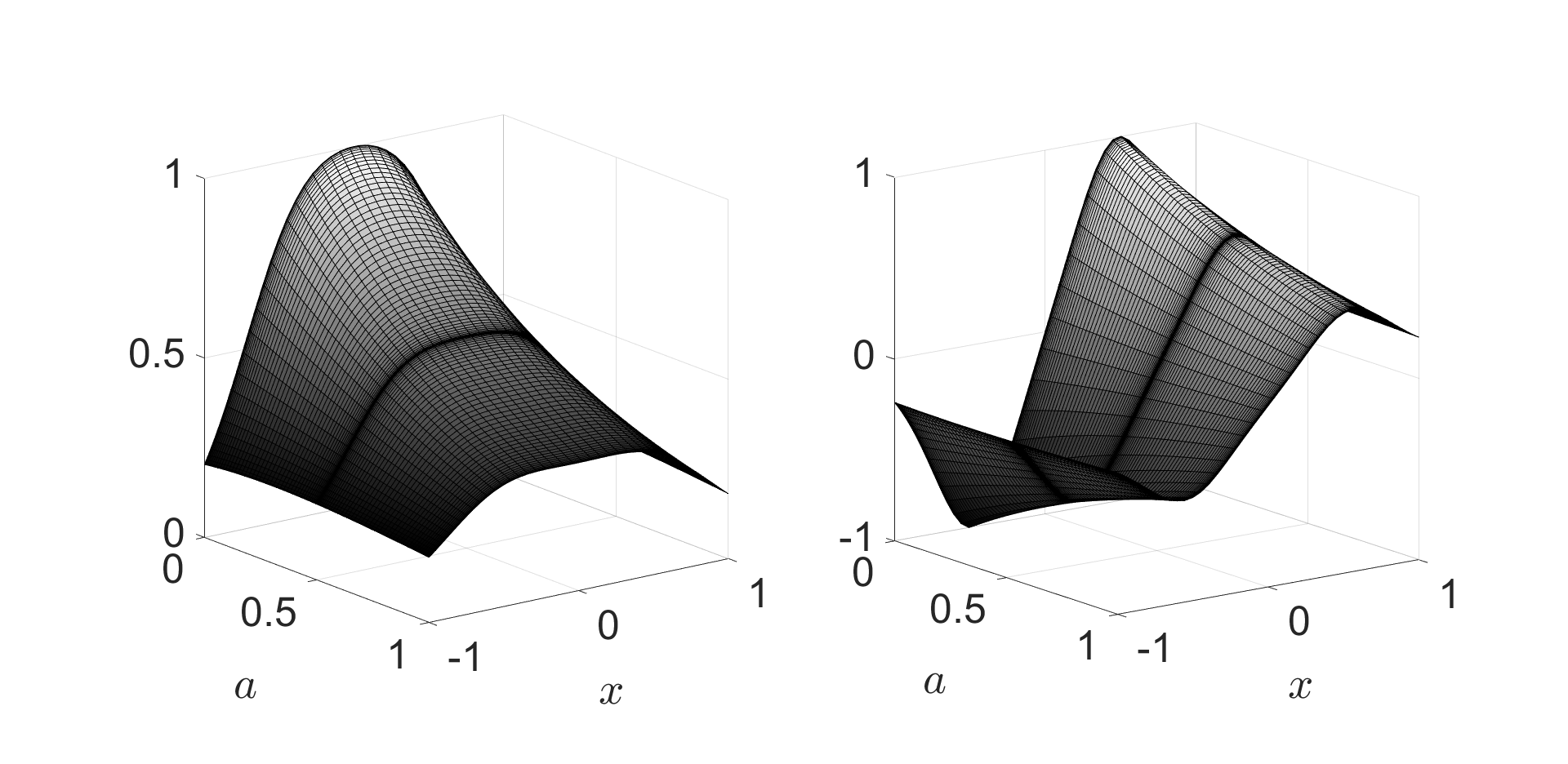}
\caption{\nameref{ex2par}: computed eigenfunctions of $\mathcal A$ relevant to $\lambda_0$ (left) and $\lambda_1$ (right) with $\beta_0=8$, $d=l=1$ and $N=M=50$.}\label{settimafig1}
\end{figure}

\medskip
\paragraph{Example 3}\label{duedim}
We consider $a^\dagger=1$, $d=1$, 
\begin{equation*}
\beta(a, x)=\beta(a)=\beta_0\chi_{\left[\frac{1}{2},\ a^\dagger\right]}(a),\quad  \mu_0(a)=\frac{1}{a^\dagger-a}, \quad \mu_1(a, x)\equiv0,
\end{equation*} 
\begin{equation*}
J(x)=\pi^{-1}e^{-\|x\|_2^2}
\end{equation*} 
and $\Omega:=\{x\in\R^2\ |\ \|x\|_2<R\}$ for $R>0$. 
By using polar coordinates the diffusion operator $\tilde{\mathcal J^x}$ on $AC_0\left([0, a^\dagger], L^2([0, R]\times[0,2\pi],\R)\right)$ reads
\begin{equation*}
\tilde{\mathcal J^x}\tilde \psi(a, r,\theta)=\int_0^R\int_0^{2\pi}J(r\cos\theta-\rho\cos\varphi,r\sin\theta-\rho\sin\varphi)\tilde\psi(a, \rho,\varphi)\rho\dd\varphi\dd \rho.
\end{equation*}
Observe that with these choices \cref{assumptionrates} holds. 
Again, we use the piecewise approach with $N=M$ (note that, being $\Omega\subset\R^2$, here $M$ represents the degree of the univariate polynomials in both directions of the tensorial approach).
\Cref{2dimfig} shows the computed spectrum of $\mathcal B$ (left) and the convergence diagram for the first two rightmost eigenvalues $\lambda_0$ and $\lambda_1$ (right) with $\beta_0=8$ and $R=1$.
\Cref{biforcazione3} shows the level curves of $\lambda_0$ as a function of $(d, l)$, which again confirm the results of \cite{ducrot2022age}.
\begin{figure}
\centering
\includegraphics[width=0.41\textwidth]{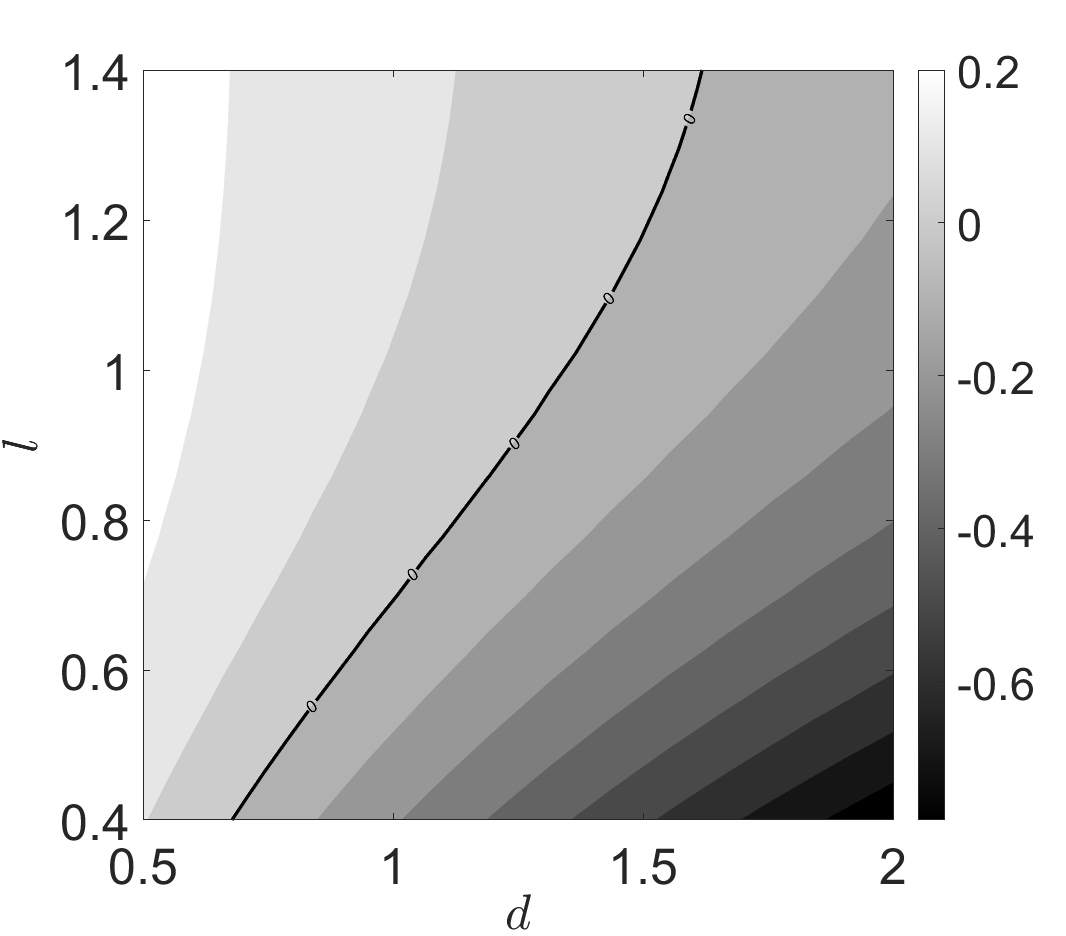}
\caption{\nameref{ex2par}: level curves for $\lambda_0$ as a function of $(d,l)$ and stability boundary (thick line) with $\beta_0=10$ and $N=M=30$.}\label{biforcazione2}
\smallskip
\centering
\includegraphics[width=0.75\textwidth]{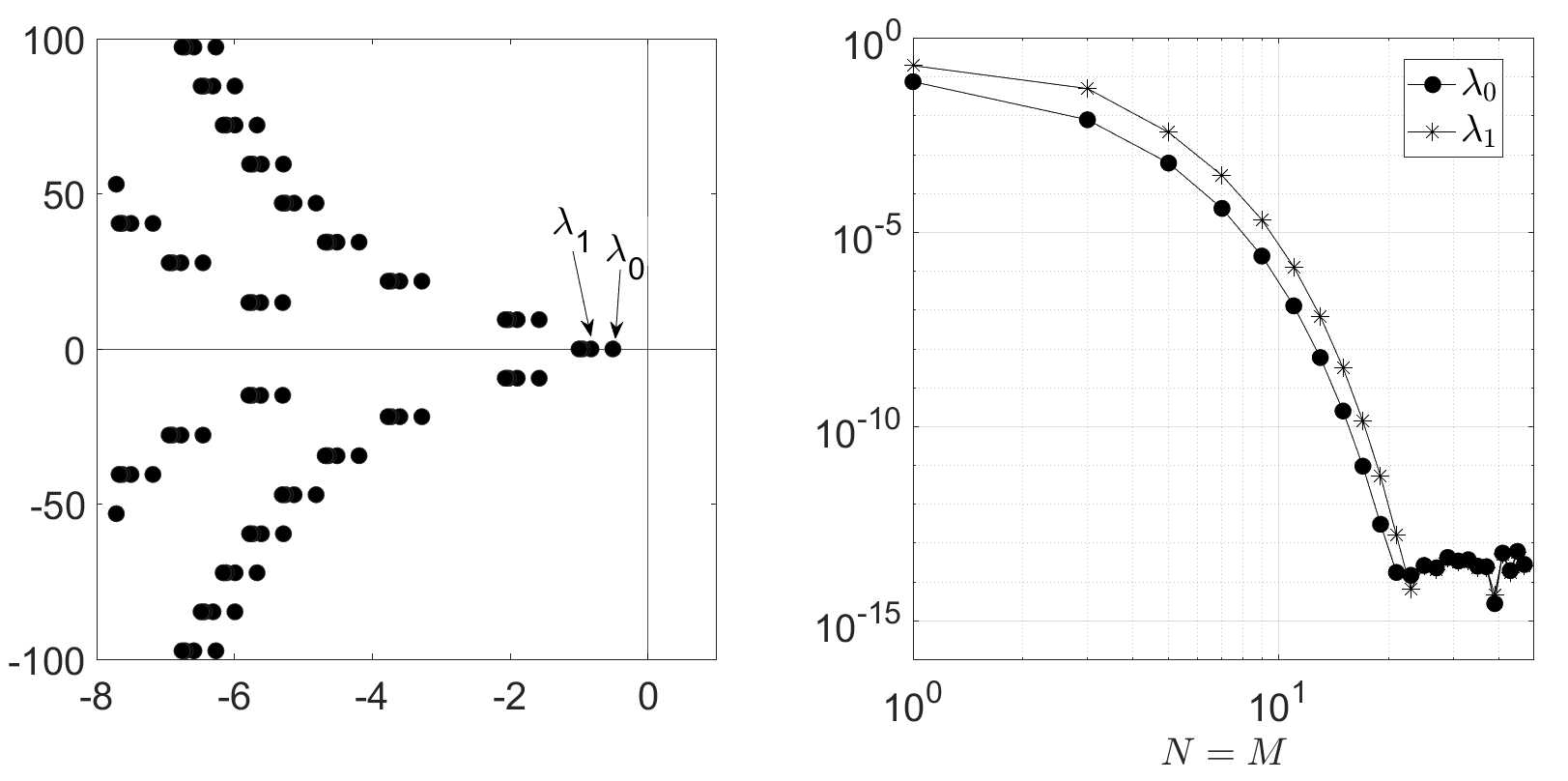}
\caption{\nameref{duedim}: spectrum of $\mathcal B$ in the complex plane computed with $N=M=50$ (left) and convergence diagram for the first two rightmost eigenvalues $\lambda_0$ and $\lambda_1$ (right) with $\beta_0=8$, $d=R=1$. The reference values $\lambda_0\approx-0.508459905995625$ and $\lambda_1\approx-0.827016421828640$ are obtained with $N=M=100$.}\label{2dimfig}
\includegraphics[width=.42\textwidth]{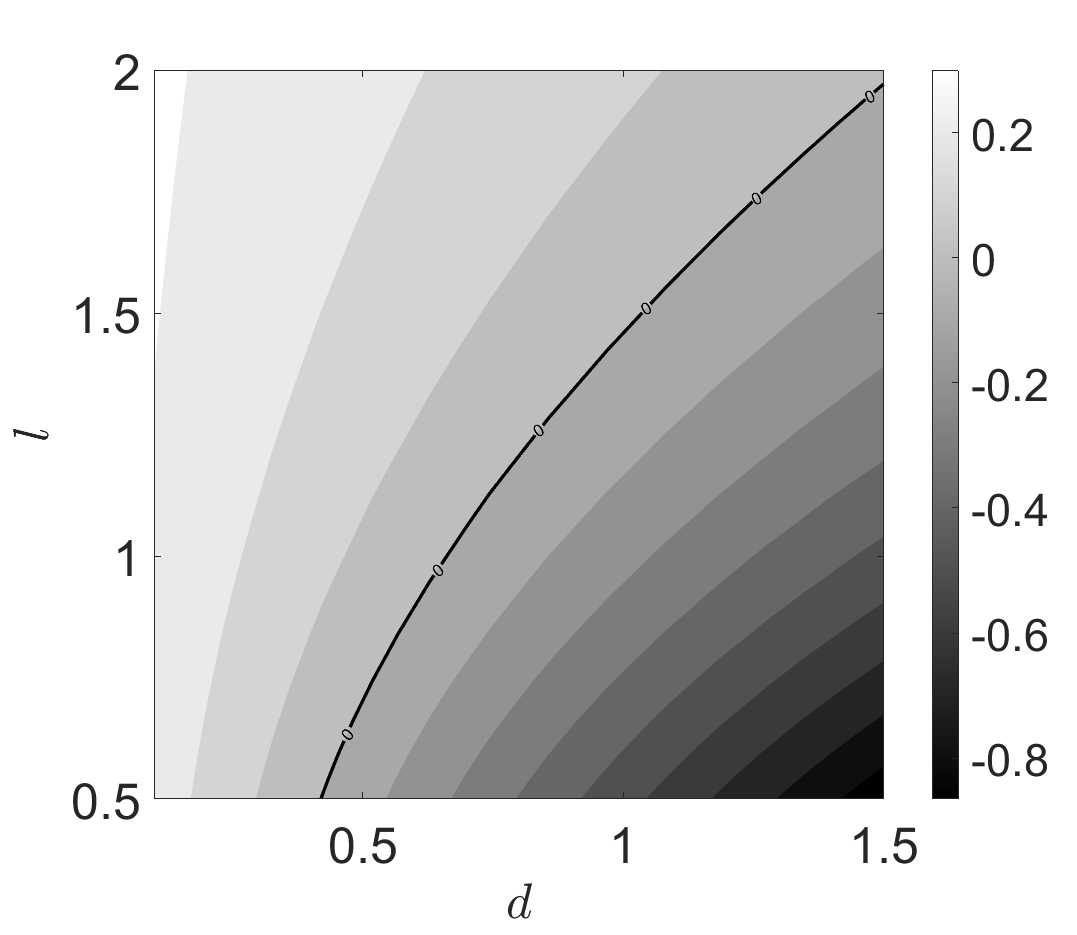}
\caption{\nameref{duedim}: level curves of $\lambda_0$ as a function of $(d,R)$ and stability boundary (thick line) with $\beta_0=10$ and $N=M=30$.}\label{biforcazione3}
\end{figure}
\section{Nonlocal diffusion of Neumann type}\label{neumann}
Now we briefly consider models with nonlocal diffusion of Neumann type, i.e., models in which the diffusion is limited to a certain region $\Omega\subset\R^n$ \cite{kang2022principal}. The prototype model is
\begin{equation}\label{diffNeumann}
\left\{\setlength\arraycolsep{0.1em}\begin{array}{rlll} 
\mathcal Du(t, a, x)&=&d\displaystyle\int_{\Omega}J(x-y)[u(t, a, y)\\[2mm]&&-u(t, a, x)]\dd y &\\[3mm]
 &&-\mu(a, x) u(t, a, x),& t>0,\  a\in[0, a^\dagger],\ x\in\Omega,\\[2mm]
u(t, 0, x)&=&\displaystyle\int_0^{a^\dagger} \beta(\hat  a, x)u(t,\hat  a, x)\dd\hat  a,\quad& t>0,\ 
x\in\Omega,\\[3mm]
u(0, a, x)&=&u_0(a, x),& a\in[0, a^\dagger],\ x\in\Omega,
\end{array} 
\right. 
\end{equation}
where $\mathcal D$ is defined as in \eqref{grad}.
Again, the semigroup generated by \eqref{diffNeumann} is not eventually compact in general, yet stability is still governed by the spectral abscissa $s(\mathcal A)$ of the relevant infinitesimal generator $\mathcal A$ (which can be derived as in \cref{Abstract setting}). Moreover, $s(\mathcal A)$ can be shown to be its principal eigenvalue under suitable assumptions \cite{kang2022principal}.

By replacing $\mathcal J-\mathcal I_U$ with $\mathcal J-\mathcal C$
for $\mathcal C\colon U\to U$ defined as
\begin{equation*}
\mathcal C \phi(a, x):=\phi(a, x)\int_{\Omega} J(x-y)\dd y
\end{equation*} and by proceeding as in sections \ref{Abstract setting} and \ref{sezioneapproccio}, we can reformulate \eqref{diffNeumann} by integration of the age-state and derive a similar numerical approximation for the spectrum of the infinitesimal generator.

Anyway, let us remark that a decomposition as in \cref{decompositioneig} has not been proved to hold true for the case of Neumann diffusion, although it is easy to see that
\begin{equation*}
\sigma_p(\mathcal K^a)+\sigma_p(\mathcal J^x-\mathcal C^x)\subseteq \sigma_p(\mathcal B),
\end{equation*}
where $\mathcal  C^x\colon L^2(\Omega,\R)\to L^2(\Omega,\R)$ is defined as
\begin{equation*}\label{C01}
\mathcal C^xg(x):=g(x)\int_{\Omega}J(x-y)\dd y.
\end{equation*}
Nevertheless, in \cite[Theorem 4.1]{kang2022principal}  it is proved that if $\lambda_0\in\sigma_p(\mathcal A)$ is associated with a positive eigenfunction $\phi$, then  $\lambda_0$ is the principal eigenvalue of $\mathcal A$. Holding this true, we can still characterize the principal eigenvalue of $\mathcal B$ under \cref{assumptionrates}.
In fact the Neumann eigenvalue problem 
\begin{equation*}
-(\mathcal J^x-\mathcal C^x)g(x)=\theta g(x)
\end{equation*}
has principal eigenvalue $\theta_0=0$ with $g_0$ constant \cite{rossi2022first}.
Thus $s(\mathcal B)=\lambda_0=\gamma_0-\theta_0=\gamma_0$ with associated eigenfunction $\psi(a, x)=f_0(a)$, where $f_0$ is the eigenfunction for $\mathcal K^a$ relevant to the principal eigenvalue $\gamma_0$.
Finally, a decomposition like \eqref{spettriequiv} holds for the discretized problem.

In light of the above results, we can proceed as in \cref{sezdim} to prove that, under \cref{assumptionrates}, the eigenvalues of $\mathcal B_{N,M}$ converge to those of $\mathcal B$ (in particular $\lambda_0$ is approximated by a converging sequence of eigenvalues).
The proof again consists in considering the separated eigenvalue problems for $\mathcal K_N^a$ and $\mathcal J_M^x-\mathcal C_M^x$. 
While the convergence analysis relevant to $\mathcal K_N^a$ remains unchanged, a slight modification is required for the space counterpart. Indeed, $\mathcal J^x-\mathcal C^x$ is not compact in general, but being bounded and self-adjoint, \cref{teospec} and \cref{corollariospec} remain valid thanks to the strong stabilty guaranteed by a (Galerkin-type) spectral projection method as the one presented in \cref{sezioneapproccio} \cite[Section 4]{chatelin2011spectral}.

\bigskip
As a numerical test, we consider the same coefficients of \nameref{ex1par} in \cref{numresults}, thus ensuring \cref{assumptionrates}. 
\Cref{nonafig} shows the computed spectrum of $\mathcal B$ (left, recall that $\gamma_0=0$) and the convergence diagram for $\lambda_0$ and $\lambda_1$. The error for $\lambda_0$ is at the machine precision already for small values of $N=M$ while the error for $\lambda_1$ decreases with infinite order (observe that the relevant eigenfunction is of class $C^{\infty}$).
\begin{figure}
\centering
\includegraphics[width=.8\textwidth]{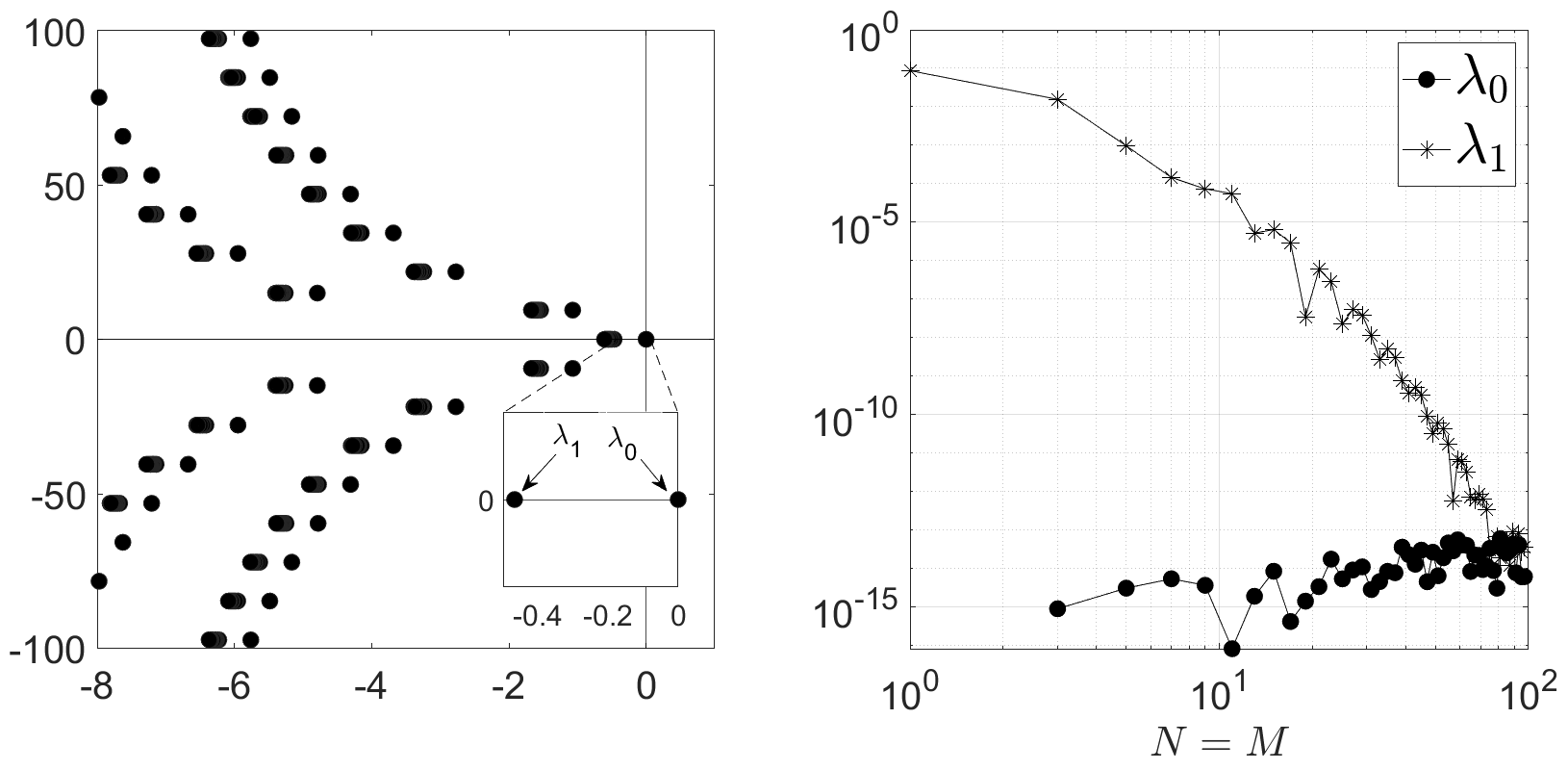}
\caption{\nameref{neumann}: spectrum of $\mathcal B$ in the complex plane computed with $N=M=100$ (left) and convergence diagram for the first two rightmost eigenvalues $\lambda_0$ and $\lambda_1$ (right)  with $d=l=1$. The reference value $\lambda_1\approx-0.467278026690667$ is obtained with $N=M=1000$.}\label{nonafig}
\end{figure}
\section{Conclusions}\label{conclusions}
In this work we proposed a numerical approach to approximate the spectrum of the infinitesimal generator of the semigroup associated to an age-structured model with nonlocal diffusion of Dirichlet or Neumann type and proved the convergence of the approximated eigenvalues in the case of separable model coefficients (\cref{assumptionrates}).

Numerical results suggest that the proposed numerical scheme is able to approximate the principal eigenvalue also in the case of parameters dependent on the spatial variable.
Future work will investigate the convergence of the method in this general case and also with regards to the monotonicity of the approximation with respect to the discretization parameters (which we did not observe in general but could be a desirable property under additional hypotheses). It is also interest of the authors to tackle the problem of the numerical approximation of $R_0$ in age-structured epidemic models involving nonlocal diffusion and also additional structuring variables \cite{breda2021bivariate, breda2021efficient, breda2020collocation,  kang2021mathematical, kang2020age}.
Another direction could be investigating the stability of models on unbounded spatial domains \cite{ruan2007spatial}.

\appendix

\printbibliography

\end{document}